\documentclass{article}
\usepackage[margin=1in]{geometry}
\usepackage[utf8]{inputenc}
\usepackage{amsfonts}       
\usepackage{nicefrac}       
\usepackage{microtype}      
\usepackage[export]{adjustbox} 

\usepackage{amsmath}
\usepackage[colorinlistoftodos]{todonotes}
\usepackage{bm}
\usepackage{breqn}
\usepackage{physics}
\usepackage{enumitem}
\usepackage{bbold}
\usepackage{epigraph}

\usepackage{csquotes}
\usepackage{hyperref}       
\usepackage{url}            
\usepackage{booktabs}       
\usepackage{amsfonts}       
\usepackage{nicefrac}       
\usepackage{microtype}      
\usepackage{color}
\usepackage{scalerel}
\usepackage[toc,page]{appendix}
\usepackage{tabularx,ragged2e,booktabs,caption}
\usepackage[font=small,labelfont=bf]{caption}
\usepackage{blindtext}
\usepackage{amssymb}
\usepackage{mathtools}
\usepackage{dsfont}
\usepackage{amsmath}
\allowdisplaybreaks
\usepackage{amsfonts}
\usepackage{amsthm}
\usepackage{cleveref}
\usepackage{float}
\usepackage{graphicx}
\usepackage{tikz}
\usepackage{subcaption}
\graphicspath{ {images/} }
\usepackage{letltxmacro}%
\usepackage{thmtools,thm-restate}
\usepackage{relsize}
\usepackage{algorithm,algpseudocode}
\usepackage{titlesec}
\usepackage{multirow}
\usepackage[algo2e]{algorithm2e}
\usepackage{algpseudocode}

\newcommand{\x}{\boldsymbol x}
\newcommand{\by}{\boldsymbol y}

\newcommand{\y}{\boldsymbol y}

\newcommand{\e}{\mathrm{e}}

\newcommand{\R}{\mathbb R}
\newcommand{\N}{\mathbb N}

\newcommand{\bu}{\boldsymbol{u}}

\newcommand{\xl}{{x_{\ell}}}

\newcommand{\vl}{{v_{\ell}}}

\newcommand{\bv}{\boldsymbol{v}}

\newcommand{\bU}{\boldsymbol{U}}

\renewcommand{\phi}{\varphi}

\newcommand{\bx}{\boldsymbol{x}}
\newcommand{\bbf}{\boldsymbol{f}}
\newcommand{\bG}{\boldsymbol{G}}

\newcommand{\bsigma}{\boldsymbol{\sigma}}
\newcommand{\bomega}{\boldsymbol{\omega}}
\newcommand{\bzeta}{\boldsymbol{\zeta}}

\DeclareMathOperator{\Acc}{Acc}

\newcommand{\weakstar}{\overset{*}{\rightharpoonup}}

\renewcommand{\:}{\mathrel{\coloneqq}}

\newcommand{\bA}{\boldsymbol{A}}

\renewcommand{\dd}{\ensuremath{\,\mathrm{d}}}

\renewcommand{\epsilon}{\varepsilon}

\newcommand{\Ia}{I_{\text{a}}}
\newcommand{\Ih}{I_{\text{h}}}

\newtheorem{theorem}{Theorem}[section]   
\newtheorem{lemma}[theorem]{Lemma}       
\newtheorem{assumption}[theorem]{Assumption}
\newtheorem{remark}[theorem]{Remark}

\newtheorem{definition}[theorem]{Definition}
\newtheorem{corollary}[theorem]{Corollary}

\crefname{assumption}{assumption}{assumptions}
\Crefname{assumption}{Assumption}{Assumptions}
\crefname{corollary}{corollary}{corollaries}
\Crefname{corollary}{Corollary}{Corollaries}

    \definecolor{golden}{RGB}{253,181,21}

\newcommand{\car}[1]{
    \draw[draw=black,fill=yellow!90!black,thick,color=yellow!90!black] (1.5,.65) circle (.03);
        \draw[draw=black,fill=red!80!black,thick,color=red!80!black] (
    0.0,.65) circle (.03);
    
    \draw[black]
        [draw=black,fill=black,rounded corners=0.0ex, thin](0,0.52)--(-0.05,0.52)--(-0.05,0.54)--(0.0,0.54);
        
        \draw[draw=black,fill=white!50!black,thick,color=white!50!black] (
    -0.15,.56) circle (.025);    
        \draw[draw=black,fill=white!50!black,thick,color=white!50!black] (
    -0.13,.54) circle (.015); 
        \draw[draw=black,fill=white!50!black,thick,color=white!50!black] (
    -0.10,.53) circle (.01);     
    
    \shade[top color=#1!70!black, bottom color=white, shading angle={135}]
        [draw=black,fill=red!20,rounded corners=0.2ex, thin](0,0.5)--(1.5,0.5)--(1.5,0.75)--(1.3,0.75)--(1,1.0)--(0.5,1)--(0.1,0.75)--(0.0,0.75)--(0,0.5);
    \shade[top color=blue!40!white, bottom color=white, shading angle={135}]
        [draw=black,fill=red!20,rounded corners=0.2ex, thin](0.35,0.8)--(0.75,0.8)--(0.75,0.95)--(0.55,0.95)--(0.35,0.8);
    \shade[top color=blue!30!white, bottom color=white, shading angle={135}]
        [draw=black,fill=red!20,rounded corners=0.2ex, thin](0.85,0.8)--(1.15,0.8)--(0.95,0.95)--(0.85,0.95)--(0.85,0.8);
    \draw[draw=black,fill=gray!50,thick] (0.3,.5) circle (.12);
    \draw[draw=black,fill=gray!50,thick] (1.2,.5) circle (.12);    
    \draw[draw=black,fill=gray!80,semithick] (0.3,.5) circle (.1);
    \draw[draw=black,fill=gray!80,semithick] (1.2,.5) circle (.1);
}

\title{Optimal Control of ODE Car-Following Models
\\ {\smaller Applications to Mixed-Autonomy Platoon Control via Coupled Autonomous Vehicles }}
\author{Arwa Alanqary \thanks{Department of Electrical Engineering and Computer Science, University of California, Berkeley, US}.
\and Alexandre M. Bayen \footnotemark[2]
\and Xiaoqian Gong \thanks{Department of Mathematics and Statistics, Amherst College, US}
\and Anish Gollakota \footnotemark[2]
\and Alexander Keimer \thanks{Department of Mathematics, University of Rostock, Germany}
\and Ashish Pandian \footnotemark[2]
}
\date{}

\begin{document}

\maketitle

\begin{abstract}
In this paper, we study the optimal control of a mixed-autonomy platoon driving on a single lane to smooth traffic flow. The platoon consists of autonomous vehicles, whose acceleration is controlled, and human-driven vehicles, whose behavior is described using a microscopic car-following model.
We formulate the optimal control problem where the dynamics of the platoon are describing through a system of non-linear ODEs, with explicit constraints on both the state and the control variables.
Theoretically, we analyze the well-posedness of the system dynamics under a reasonable set of admissible controls and establish the existence of minimizers for the optimal control problem.
To solve the problem numerically, we propose a gradient descent-based algorithm that leverages the adjoint method, along with a penalty approach to handle state constraints.
We demonstrate the effectiveness of the proposed numerical scheme through several experiments, exploring various scenarios with different penetration rates and distributions of controlled vehicles within the platoon.
\end{abstract}

\section{Introduction}
Congestion in traffic networks is a long-standing challenge that has widespread societal and economic implications.
During the past decades, a significant amount of research has been done to understand the sources and mitigate the effects of traffic congestion. 
Different phenomena have been understood as contributors to congestion, the presence of physical bottlenecks, like ramp merging and lane merging, contributes to congestion by limiting the flow of vehicles. 
However, even in the absence of these physical bottlenecks, the phenomenon of stop-and-go traffic waves is a well-known issue \cite{treiterer1974hysteresis, suh2016empirical}.
These waves result from the inherent instability in traffic flow, which is influenced by the collective behavior of human drivers. 
Small disturbances, such as sudden braking or acceleration by one driver, can propagate through the traffic stream, leading to (the amplification of) congestion.
The seminal work of Sugiyama et al \cite{sugiyama2008traffic} provides clear evidence and experimental demonstration of the formation and propagation of stop-and-go traffic waves due to human car-following instability in a circular platoon. 

While physical bottlenecks can be tackled by infrastructure fixes, mitigating traffic congestion due to stop-and-go waves requires influencing human driving behavior, including car following behavior. 
Major efforts have been established in studying and implementing such mitigation. 
Ramp metering \cite{hou2007freeway, papageorgiou2002freeway} is one of the classical approaches to mitigate stop-and-go congestion, and it is still in use today. 
Variable speed limit control is another technique that modifies speed limits on the road to influence traffic flow and reduce congestion \cite{papageorgiou2008effects, hegyi2005optimal}. 

Recent advances in vehicle automation have created a promising new paradigm for effective traffic control. 
Modern vehicles are increasingly equipped with sophisticated sensing, communication, and automation technologies that enable them to interact with their environment and with each other in real time. 
This paves the way for a gradual transition from primarily human-driven traffic to a mixture of autonomous and human-driven vehicles, a regime referred to as mixed-autonomy traffic. 
Even at low penetration rates, the presence of autonomous vehicles (AVs) offers a promising opportunity for efficient control of transportation systems, as they can serve as actuators for regulating traffic flow. 
By controlling the acceleration and the position of AVs within a platoon, they can smooth traffic, reduce congestion, and improve overall road safety, even in environments where human drivers remain predominant.

This idea has been tested in simulation and experimentally over the years with focus on improving traffic energy efficiency, throughput, average speed, among other metrics. 
%
%
We refer to some of the earlier work on intelligent cruise control and platoon control in \cite{ioannou1993intelligent,darbha1999intelligent,davis2004effect,rajamani1998design,martinez2007safe}, and more recent developments in \cite{delle2019feedback,kerner2018physics,stern2018dissipation, jia2020lstm,vinitsky2023optimizing,lichtle2022deploying, fu2024kernel} among many others. 
This control paradigm has shown considerable success in attenuating stop-and-go waves and improving traffic conditions.  
For example, Stern et al. \cite{stern2018dissipation} demonstrated, in a field experiment, that introducing a single AV equipped with a proportional-integral controller can dissipate stop-and-go waves in a circular platoon. 
More recently, Lee et al. \cite{lee2024traffic} demonstrated the effectiveness of utilizing AVs as Lagrangian traffic actuators to reduce traffic flow instability, based on a large-scale field experiment in which 100 AVs were deployed in dense highway traffic.

In the present work, our aim is to examine the potential and limitation of this control paradigm and to create a baseline for the performance of AVs in attenuating stop-and-go waves by means of optimal control.
We consider the simple setup of a single-lane car following behavior and assume perfect knowledge of downstream traffic conditions. Under such idealistic assumptions, we examine the performance limit that can be achieved by introducing controlled vehicles in human-driven traffic. 
With this goal in mind, we formulate the task of car-following in mixed autonomy traffic as an optimal control problem (OCP) over a fixed time horizon and with pro-specified leader trajectory. 
We model the non-linear dynamics of a mixed-autonomy platoon where the AVs are controlled in their acceleration and the human-driven vehicle's acceleration is governed by a car-following model. 
We then formulate an objective functional that considers the performance of the AVs as well as all the HVs that follow. 
We theoretically analyze the proposed problem, involving proving the existence of minimizers under reasonable assumption on the car following model and reasonable admissible control set. 
We then propose a numerical method to solve the optimal control problem that utilizes a gradient descent algorithm with adjoint-based formulation of the required gradients.

\subsection{Traffic Modeling} Central to our modeling assumptions in this work is the choice of the car-following model which describes the driving behavior of vehicles on single lane traffic. Since the early 1950s, researchers have been developing such models and using them to simulate and understand various traffic phenomena. These models try to capture how a vehicle adjusts its acceleration in response to the input it observes of the leading vehicle to maintain a desired speed and safety distance. These observations often include the space gap, relative velocity and the leader vehicle's velocity. Among these models, the most commonly studied and used models are the intelligent driver model \cite{treiber2000congested}, and the class of optimal velocity models initially introduced by Bando \cite{bando1995dynamical}. For a more comprehensive overview, see \cite{ahmed2021review, li2012microscopic}. Despite their popularity in various applications, limited theoretical guarantees exist for the well-posedness of these car following models dynamics (see for instance \cite{albeaik_2022_limitations}). 

In this study, we use the \textit{Bando-follow the leader} model (Bando-FtL) which is a mixture of the original optimal velocity model (OVM) proposed by Bando and an additional \textit{follow the leader} term. This model has provable well-posedness guarantees and explicit bounds on the acceleration under very mild conditions on the lead vehicle trajectory as outlined in \cite{gong2022well}. These features are desirable for numerical simulations and necessary for proving the well-posedness of the optimal control problem considered, (i.e., the existence of a minimizer of the optimal control problem based on these dynamics). 

\subsection{Optimization based traffic control through autonomous vehicle}  
Optimal control is a prevalent tool for modeling and solving various tasks in the domain of automated vehicles, including car-following, trajectory optimization, and obstacle avoidance. 
The literature on this topic is extensive, but we will narrow our focus to advancements that employ comparable formulations and solution techniques to those introduced in this paper. 
Wang et al. \cite{wang2021leading} studied the controllability and stabilizability of a mixed-autonomy platoon using a single AV. 
They use the optimal velocity model (OVM) for both human-driven traffic and the AV's car following behavior. 
The AV has an additive acceleration that is being controlled using a state feedback controller. The controller design relies on the linearization of the system dynamics. 
Similar techniques were proposed for circular platoons in \cite{wang2020controllability, zheng2020smoothing}.

Wang et al. \cite{wang_2022_optimal} formulated the problem of smoothing single-lane mixed-autonomy traffic flow through optimal control of AVs. 
The dynamics combine the intelligent driver model (IDM) for human driven vehicles with a modified OVM with positive additive acceleration for the AVs. 
In contrast, our work uses the Bando-FtL model which enjoys well-posedness of the model independent on the vehicles behavior ahead.
Without these properties, analyzing the related optimal control problems becomes intricate or even impossible as existence of solutions might depend on the choice of acceleration profiles.
For the AVs, we employ a more representative control parametrization compared to using CFMs with additive acceleration as suggested in \cite{wang2021leading, wang_2022_optimal}. 
While the latter approach relies on the car-following dynamics for collision avoidance --- which is not rigorously demonstrated --- it comes at the expense of restricting the AVs' behavior.
Further, with the objective function proposed in \cite{wang_2022_optimal}, which only penalizes the perturbations of the AVs' speed, this modeling choice can lead to unrealistic vehicle dynamics. 
This can occur when the additive acceleration results in an arbitrarily small headway requiring an unrealistically high deceleration by the CFM to avoid collision. 
In our approach, we explicitly handle the safety constraints leading to natural bounds on the vehicle's acceleration independently of the objective functional. 
Given these modeling choices, we employ a solution approach similar to \cite{wang_2022_optimal}, leveraging gradient descent and the adjoint representation of the gradient.
%

%

A related line of research applies model predictive control (MPC) to car-following tasks \cite{lin2017simplified, wu2022hierarchical, schmied2015nonlinear}. 
In these controllers, the planning layer is often cast as a receding horizon optimal control problem (OCP) with a single AV. 
However, these formulations do not model human traffic and only respond to a leader trajectory. 
This enables the formulation of the OCP as a linearly constrained quadratic program which can be solved efficiently making it well-suited for real-time applications. 

For general results on optimal control of ODEs and extensions to weaker forms of the related Pontryagin's maximum principle, we refer the reader to \cite{Mangasarian1966,Maurer1981,Sussmann1997,Sussmann,angrisani2025}.


\subsection{Outline} In \cref{sec:setup} we state the notation and main definitions and introduce the problem. 
In \cref{sec:traffic_modeling}, we introduce the traffic model of choice. 
In \cref{sec:analysis}, we discuss the existence of solution to the proposed optimal control problem with the specified traffic model.
In \cref{sec:method} we introduce a gradient descent algorithm based on the adjoint formulation to solve the proposed optimal control problem. 
In \cref{sec:experiemnts}, we present the numerical results on different scenarios. 

\section{Problem Setup} 
\label{sec:setup}
We consider a system of mixed-autonomy platoon of vehicles driving on a single lane. 
Let \(N\in\N_{\geq1}\) and \(M\in\N_{\geq1}\) be the number of human-driven and autonomous vehicles, respectively. 
Additionally, let the sets
\[
\Ia\subset \{1,\ldots,N+M\}:\ |\Ia|=M,\ \Ih\:\{1,\ldots,N+M\}\setminus \Ia
\]
denote the indices of \textbf{autonomous} and \textbf{human-driven} vehicles, respectively. We set 
\[
I\:\Ia\cup\Ih=\{1,\ldots,N+M\}.
\]
Given a finite time horizon \(T \in \R_{>0}\), we define, for a given time \(t\in[0,T]\), the vectors
\[\bx(t) \in \R^{|I|},\ \bv(t) \in \R^{|I|}\]
where \(\bx(t)\) are the positions of the vehicles and \(\bv(t)\) are their velocities.
To simplify the notations, we define the time-dependent control vector for all vehicles as
\[\bu(t) \in \R^{|I|},\ t \in [0,T]\]
which represents the acceleration of the autonomous vehicles (AVs), and take zero value for human-driven vehicles (HVs) as they are not controlled.
Lastly, all positions and velocities have an initial state at time \(t=0\)
\begin{align*}
    \bx_{\circ} \in \R^{|I|},\ & \bv_{\circ} \in \R_{\geq0}^{|I|}
\end{align*}

Having introduced the main notation, we proceed with some additional modeling assumptions.
\begin{assumption}[Given leader's trajectory] 
\label{ass:leader_trajectory}
We assume that the leading vehicle's trajectory is given and satisfies 
\begin{equation}
\xl\in W^{2,\infty}((0,T)): \ \dot{x}_{\ell}(t)\geq 0\ \quad \forall t\in[0,T]
\label{eq:leader_trajectory}
\end{equation}
\end{assumption}
This ensures that the leader's trajectory is given such that the lead vehicle's velocity is always non-negative and smooth enough so that it also admits an expression for the acceleration \(\ddot{x}_{\ell}\).
%

\begin{assumption}[Minimal initial distance between cars]
\label{ass:min_dist}
For a given vehicle length \(l \in \R_{>0}\), we assume that there exists \(d_{\circ}\in\R_{>0}\) such that 
\begin{equation}
    \bx_{\circ,i}-\bx_{\circ,i+1}-l > d_{\circ} \ \forall i\in\{1,\ldots,N+M-1\} \wedge \xl(0)-\bx_{\circ,1}-l > d_{\circ}
\label{eq:min_dist}
\end{equation}
\end{assumption}
This assumption ensures an initial space headway between all vehicles is above a some distance \(d_{\circ}>0\). 

\begin{assumption}[Reasonable control set]
\label{ass:control_set}
For a given \(a_{\min}\in\R_{<0}\) and \(a_{\max}\in\R_{>0}\) the controls vector satisfies \(\bu \in \bU_{a_{\min}}^{a_{\max}}\) with
\begin{equation}
\begin{aligned}
\bU_{a_{\min}}^{a_{\max}}\: &  \Big\{\bu\in L^{\infty}\big((0,T);\R^{|I|}\big):\bu_{i}\equiv 0\ \forall i\in\Ih \\
& a_{\min}\leq \bu_{i}(t)\leq a_{\max} \ \wedge\ \bv_{\circ,i}+\int_{0}^{t}\!\!\!\bu_{i}(s)\dd s\geq 0\ \forall (t,i)\in[0,T]\times \Ia\Big\}.
\end{aligned}
\label{eq:control_set}
\end{equation}
\end{assumption}
That is, we control the AVs' acceleration such that it lie between \(a_{\text{min}}<0\) and \(a_{\text{max}}>0\) and all AVs drive at a non-negative velocity. 
The HVs' acceleration is not controlled as expressed in \cref{ass:control_set} with the condition \(\bu_{i}\equiv 0\ \forall i\in\Ih\).

We now define the system of ordinary differential equations (ODEs) that describes the evolution of positions and velocities of all vehicles, given that the controls \(\bu\) are applied. 
\begin{definition}[System dynamics]
\label{def:dynamics}
Given \cref{ass:leader_trajectory,,ass:min_dist,,ass:control_set}, and given the acceleration function 
\[
\Acc \colon \big(\mathbb{R}^2 \times \mathbb{R}_{\geq 0}^{2}\big) \mapsto \mathbb{R},\]
we consider for \(t\in[0,T]\) the following system of initial value problems
\begin{equation}
\begin{aligned}
\bx_{0}(t)&=\xl(t)\\
\bv_{0}(t)&=\vl(t)\\
   \dot{\bx}_{i}(t) &=\bv_{i}(t),& i \in I\\
   \dot{\bv}_{i}(t)&=\Acc\big(\bx_{i}(t), \bx_{i-1}(t), \bv_{i}(t),\bv_{i-1}(t)\big) &i\in \Ih \\
   \dot{\bv}_{i}(t)&= \bu_i(t) &i\in \Ia\\
   \bx(0)&=\bx_{\circ}\\
   \bv(0)&=\bv_{\circ}
\end{aligned}
\label{eq:dynamics}
\end{equation}
\end{definition}

\begin{remark}[The position-velocity vector] 
The vectors \(\bx(t)\) and \(\bv(t)\) give the positions and velocities of  \(M+N \) vehicles at time $t \in [0, T]$. 
For convenience of notation, and considering that the leading vehicle is positioned in front of the first vehicle $i = 1$ we use, according to \cref{ass:leader_trajectory}, the notation $\bx_0(t) = \xl(t)$ and $\bv_0(t) = \vl(t)$. 
%
\end{remark}

\begin{remark}[The car-following dynamics]
We point out that the controls acting on AV \(i \in \Ia\) can only affect the positions and velocities of their following vehicles. 
That is, the position and velocity of vehicle $i$, \(\bx_i\) and \(\bv_i\), only depend on the controls \(\bu_{k}\), with $k \leq i$ as the system of ODEs is only \textit{one-sided} coupled.
\end{remark}

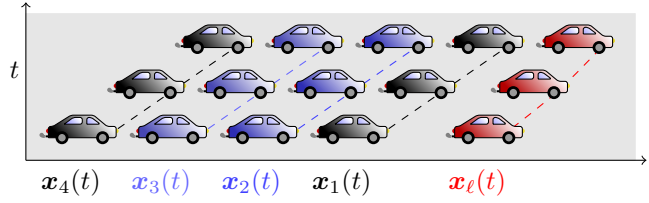
\begin{figure}
    \centering
\begin{tikzpicture}[scale=0.6]
 \draw[fill,color=gray!20!white](-6.5,-1)--(-6.5,2.25)--(7,2.25)--(7,-1)--cycle;
 \draw[->] (-6.5,-1) -- (-6.5,2.5);
\draw (-6.75,1) node {\(t\)};
\draw[->] (-6.5,-1) -- (7.25,-1);
\draw[dashed,black] (1.25,-0.5) -- (4.5,1.75);
\draw[dashed,blue!75!white] (-.75,-0.5) -- (2.5,1.75);
\draw[dashed,blue!60!white] (-2.75,-0.5) -- (.5,1.75);
\draw[dashed,black] (-4.75,-0.5) -- (-1.5,1.75);
\draw [dashed,red] plot [smooth] coordinates { (4,-0.5) (5.25,0.5) (6.25,1.5)};
\draw (3.5,-1) node[below] {\textcolor{red}{\(\bx_{\ell}(t)\)}};
\draw (0.5,-1) node[below] {\textcolor{black}{\(\bx_{1}(t)\)}};
\draw (-1.5,-1) node[below] {\textcolor{blue!75!white}{\(\bx_{2}(t)\)}};
\draw (-3.5,-1) node[below] {\textcolor{blue!60!white}{\(\bx_{3}(t)\)}};
\draw (-5.5,-1) node[below] {\textcolor{black}{\(\bx_{4}(t)\)}};
\begin{scope}[shift={(1.5,0)}]\car{black}\end{scope}
 \begin{scope}[shift={(5,1)}]\car{red}\end{scope}
 \begin{scope}[shift={(3,1)}]\car{black}\end{scope}
 \begin{scope}[shift={(4,0)}]\car{red}\end{scope}
 \begin{scope}[shift={(0,-1)}]\car{black}\end{scope}
 \begin{scope}[shift={(3,-1)}]\car{red}\end{scope}
 \begin{scope}[shift={(-2,-1)}]\car{blue!75!white}\end{scope}
 \begin{scope}[shift={(-0.5,0)}]\car{blue!75!white}\end{scope}
 \begin{scope}[shift={(1,1)}]\car{blue!75!white}\end{scope}
  \begin{scope}[shift={(-4,-1)}]\car{blue!60!white}\end{scope}
 \begin{scope}[shift={(-2.5,0)}]\car{blue!60!white}\end{scope}
 \begin{scope}[shift={(-1,1)}]\car{blue!60!white}\end{scope}
 \begin{scope}[shift={(-6,-1)}]\car{black}\end{scope}
 \begin{scope}[shift={(-4.5,0)}]\car{black}\end{scope}
 \begin{scope}[shift={(-3,1)}]\car{black}\end{scope}
 \end{tikzpicture}
\caption{The leader \textcolor{red}{\(\bx_{\ell}(t)\equiv \bx_{0}(t)\)} with its dynamics determined by the acceleration \(u_{\text{lead}}(t)\) and the following cars \textcolor{black}{\(\bx_{1}(t)\)}, \textcolor{blue!75!white}{\(\bx_{2}(t)\)},  \textcolor{blue!60!white}{\(\bx_{3}(t)\)}, \textcolor{black}{\(\bx_{4}(t)\)} with its dynamics governed system \cref{eq:dynamics}. Here, cars \textcolor{black}{\(\bx_{1}(t)\)} and \textcolor{black}{\(\bx_{4}(t)\)} are autonomous and cars \textcolor{blue!75!white}{\(\bx_{2}(t)\)} and \textcolor{blue!60!white}{\(\bx_{3}(t)\)} are human-driven. }
    \label{fig:car_following_multi}
\end{figure}

Let us for now assume that for \(\bu\in \bU_{a_{\min}}^{a_{\max}}\) the solution of \cref{eq:dynamics} is uniquely defined on the time horizon \(T\in\R_{>0}\). Then, we can write the solution as
\begin{equation}
(\bx[\bu],\bv[\bu])\in W^{2,\infty}\big((0,T);\R^{|I|}\big)\times W^{1,\infty}\big((0,T);\R^{|I|}\big),
\label{eq:bx[bu]}
\end{equation}
indicating the functional dependency of the solution \((\bx,\bv)\) with regard to the control \(\bu\). 
The fact that there is such a functional dependency is not obvious and needs to be shown (we refer to \cref{sec:traffic_modeling} for the respective discussion).

Having stated the system dynamics and introduced the notation for its solution, we now need to restrict the control set further so that the AVs do not collide with their preceding vehicles, a condition not yet guaranteed by the control set in \cref{ass:control_set}. 
To see this, consider the control choice \(\bu\in \bU^{a_{\max}}_{a_{\min}}\) with maximal acceleration, which would result in a linearly increasing velocity for the AVs over time. For time \(T\) large enough, an AV following an HV will inevitably collide with this proceeding HV if its velocity remains bounded.
Indeed, the admissible control set is defined to incorporate a safety distance between the AVs and their preceding vehicles:
\begin{definition}[Admissible control set incorporating a safety distance]
\label{def:U}
With the notation in \cref{eq:bx[bu]} in mind, we define the set of admissible controls \(\bU\) as 
\begin{equation}
\bU \: \Big\{\bu\in\bU_{a_{\min}}^{a_{\max}}:\ \bx_{i-1}[\bu](t)-\bx_{i}[\bu](t)-l\geq d_{\text{safe}}\ \forall(t, i) \in[0, T] \times \Ia\Big\}
\label{eq:U}
\end{equation}
for a given vehicle length \(l \in \R_{>0},\) a saftey distance \(d_{\text{safe}} \leq d_{\circ} \in\R_{>0}\) where \(d_{\circ}\) is as defined in \cref{ass:min_dist}, and the control set \(\bU_{a_{\min}}^{a_{\max}}\) as defined in  \cref{ass:control_set}.
\end{definition}

\begin{remark}[The admissible control set]
Note that the defined set \(\bU\) could be empty if \(a_{\min}\in\R_{<0}\) is not small enough. 
For example, at the initial time, if the vehicle ahead is standing still and the following AV is approaching at a high velocity, the minimal deceleration might not be enough to avoid collision.
However, if we assume that \(a_{\min}\) is significantly small, the set can be proven to be non trivial as will be detailed in \cref{cor:well-posedness_system}.
\end{remark}

\begin{remark}[Safety distance for human-driven vehicles ]
Note that the defined set \(\bU\) does not encode a safety distance between the HVs and their preceding vehicles. 
This safety distance should be guaranteed by the selected car following model. 
We will discuss this in more details in \cref{sec:traffic_modeling} and give lower bounds on the headway of HVs (see \cref{th:well_posedness_Bando}) for the specific car-following model of choice (the Bando-Follow the leader model). 
\end{remark}

Given the above definitions and modeling assumptions, we can formally define the optimal control problem (OCP) of interest. 
The goal is to find an optimal control vector $\bu^*$ minimizing a certain objective function $J(\bx, \bv, \bu)$. 
For now, we abstract the objective function and introduce the general optimal control problem in \cref{def:ocp}. 
\begin{definition}[Objective function and the related optimal control problem]
\label{def:ocp}
Let \(L\in C^{1}\big(\R^{|I|}\times \R^{|I|}\times\R^{|I|};\R_{\geq0}\big)\) be the \textbf{running cost} and assume that \(L(a,b,\cdot)\) is convex for all \((a,b)\in\R^{2|I|}\). 
Let \(S\in C^{1}\big(\R^{|I|}\times\R^{|I|};\R_{\geq0}\big)\) be the \textbf{final cost}. 
We call
\begin{align}
\tag{OCP}
\label{eq:ocp}
    \inf_{\bu\in\bU}\!\! J(\bx[\bu],\bv[\bu],\bu)&\:\int_0^T\!\!\! L\big(\bx[\bu](t),\bv[\bu](t),\bu(t)\big) \dd t + S\big(\bx[\bu](T),\bv[\bu](T)\big)\\
    \text{with }& (\bx[\bu],\bv[\bu])\text{ as in \cref{eq:dynamics} and \(\bU\) as in \cref{def:U}}\notag
\end{align}
the \textbf{optimal control problem (OCP)}  considered.
\end{definition}

\begin{remark}[The optimal control problem and the assumptions on \(L \text{ and } S\)]
The assumptions on the proposed objective functions are rather simplistic. 
The convexity assumption in the last argument of \(L\) ensures the existence of a minimizer (as will be discussed in \cref{th:existence_minimizer}). 
This assumption is not restrictive, as we will show in \cref{sec:experiemnts}, where we instantiate meaningful optimal control problems that satisfy this assumption by using the weighted \(L^{2}\)-norm of the control vector.
Additionally, we assume that both \(L\) and \(S\) are continuously differentiable. This assumption is useful for deriving the adjoint system and the first-order optimality condition for the problem. 
%
\end{remark}

\section{Traffic Modeling}
\label{sec:traffic_modeling}
So far, we have not specified the car-following model we are using, i.e., the \(\Acc\) function in \cref{def:dynamics}. 
This is what we will introduce in the following. 
We will aim for the well-known Bando-Follow-the-Leader model (Bando-FtL) \cite{DelleMonache2019}, however, one should emphasize that any second order car-following model would be suitable as long as it has properties like being collision free (i.e.~well-posed) for all admissible leader trajectories. 
This cannot be assured for all second order car-following models which currently exist and are commonly used. 
For instance, refer to the problems with the IDM (intelligent driver model) in \cite{albeaik_2022_limitations}.
For this reason, we restrict ourselves for now to the named Bando-FtL model as its properties are well-understood (see \cite{gong2022well}).
In this section, we define the Bando-FtL model and give some results on its well-posedness as well as lower bounds on the velocity and the headway. 
These bounds will be crucial when we later prove the well-posedness of our optimal control problem.
\begin{definition}[Bando-Follow-the-Leader model]
\label{defi:bando_ftl}
For a given vehicle length \(l\in\R_{>0},\) define the set 
\[
A\:\big\{(x,\xl,v,\vl)\in\R^{4}:\ \xl-x-l>0\big\}.
\]
With parameters \((\alpha,\beta)\in\R_{>0}^{2}\), and a function 
\begin{equation}
    V\in C^{1}(\R_{>0};\R_{\geq0})\cap L^{\infty}(\R_{\geq0})\ \text{ satisfying } V'\geq 0, 
    \label{eq:assumption_V}
\end{equation} 
the Bando-follow-the-leader model is instantiated by the definition of the acceleration
\begin{align*}
\Acc :
\begin{cases}
A & \mapsto\R\\
(x,\xl,v,\vl) & \mapsto  \alpha \Big(V(\xl-x-l)-v\Big) + \beta\frac{\vl-v}{(\xl-x-l)^{2}}.
\end{cases}
\end{align*}
\end{definition}
\begin{remark}[The meaning of the involved constants and parameters]
In \cref{defi:bando_ftl} of the Bando-FtL model, the parameter \(l\) represents the length of the vehicle (this is the same parameter used in \cref{ass:min_dist} and  \cref{def:U}). 
\(\alpha\) and \(\beta\) are sensitivity parameters and \(V\) is the optimal velocity function. 
In general, the optimal velocity function \(V\) is monotonically increasing with respect to the space headway, i.e., \(V'\geq 0\). 
A common choice of this optimal velocity function, which we will be using throughout this work, is given as
\begin{equation}
    V(h) = v_{\max} \tfrac{\tanh(h-d_{\text{s}})+\tanh(l + d_{\text{s}})}{1+\tanh{(l + d_{\text{s}})}}, \ h\in\R_{\geq 0},
    \label{OV_fun}
\end{equation}
where \(v_{\max}\in\R_{>0}\) is the maximum allowable velocity of the vehicle, and $d_{\text{s}}\in\R_{>0}$ is a safety distance parameter.
\end{remark}

Having stated the dynamics for the Bando-FtL model, we now give its well-posed in the scenario where there is a leader satisfying \cref{ass:leader_trajectory} and the follower's acceleration, and thus velocity and position, are determined by the model. 
This is enough for the well-posedness of the entire system of car following models in \cref{def:dynamics} as the dependencies of the different HVs is only in one direction. That is, the acceleration of vehicle \(i+1\in I_{h}\) depends on vehicle \(i\) and itself only (as we model the HVs by the car-following models). This dependency is not true for the AVs as their acceleration is a result of the optimization and as such can depend on all states in the system.
\begin{theorem}[Well-posedness of the Bando-FtL and some bounds]
\label{th:well_posedness_Bando}
Recall \cref{ass:leader_trajectory,,ass:min_dist,,ass:control_set} and let \(\Ia=\emptyset\) and \(\Ih=\{1\}\). Then, the system dynamics in \cref{def:dynamics} admits a unique solution \(\bx_{1}\in W^{2,\infty}((0,T))\). For the constants
\begin{align}
A&\:-\bv_{0}-\alpha T\|V\|_{L^{\infty}(\R_{>0})}+\alpha (\xl(0)-\x_{\circ,1}-l)-\beta\tfrac{1}{\xl(0)-\bx_{\circ,1}-l}\\
d_{\min}&\:\tfrac{A+\sqrt{A^{2}+4\alpha\beta}}{2\alpha}\label{defi:d_min}\\
B&\:\tfrac{\alpha d_{\min}^{2}+\beta}{d_{\min}^{2}}\label{defi:B}
\end{align}
the solution satisfies the following for all \( t\in[0,T]\): 
Lower bound on the headway: 
    \begin{equation}
    \xl(t)-\bx_1(t)-l\geq d_{\min},\label{eq:d_min}
    \end{equation}
Lower and upper bounds on the velocity: 
    \begin{equation}
    \bv_{\circ, 1}\exp\big(-Bt\big) \leq \dot{\bx_1}(t)\leq \max\Big\{\bv_{\circ, 1},\|V\|_{L^{\infty}(\R)}+\tfrac{\beta}{\alpha}\tfrac{\vl(t)}{d_{\min}^{2}}\Big\},\label{eq:lower_upper_bound_velocity}
    \end{equation}
Lower and upper bounds on the acceleration: 
    \begin{small}
    \begin{align}
       -B\max\Big\{\bv_{\circ,1},\|V\|_{L^{\infty}(\R)}+\tfrac{\beta}{\alpha}\tfrac{\vl(t)}{d_{\min}^{2}}\Big\} \leq \ddot{\bx}(t)\leq \alpha\|V\|_{L^{\infty}(\R_{\geq 0})}-\alpha \bv_{\circ,1}\e^{-Bt}+\beta\tfrac{\vl(t)}{d_{\min}^{2}}.\label{eq:bounds_acceleration}
    \end{align}
    \end{small}
\end{theorem}
\begin{proof}
The proof can be found in \cite{gong2022well}.
\end{proof}

Next, we show that for small enough \(a_{\min}\), the maximal deceleration, we obtain that the admissible control set is not empty.
\begin{corollary}[\(\bU \neq \emptyset\) and well-posedness of the system dynamics]
\label{cor:well-posedness_system}
For \(\bU\) as in \cref{def:U} with \(Acc\) as in \cref{defi:bando_ftl} and for sufficiently small \(a_{\min}\in\R_{<0}\) the set of admissible controls is non empty, i.e.,
\[
\exists a_{\min}\in\R_{<0}:\ \bU\neq\emptyset.
\]
Additionally, for such an \(a_{\min}\in\R_{<0}\) and \(\bu\in \bU\) there exists a unique solution of the system dynamics as stated in \cref{def:dynamics}.
\end{corollary}
\begin{proof}
As the system dynamics are only one-sided coupled, we can perform an induction over \(k\in I\).
For \(k=1\) we have either
\begin{description}
\item[\(1\in\Ia\):] 
    Then, we need to show that there exists a control that prevents vehicle \(1\) from colliding with, or getting too close to, its leading vehicle. Following \cref{def:U} we want to ensure that \(\xl(t)-\bx_{1}[\bu](t)-l\geq d_{\text{safe}} \, \forall t\in[0,T]\). One such control involves applying heavy braking initially to bring the AV to zero velocity, followed by maintaining zero acceleration for the remainder of the time horizon. We show the existence of such controller in the ``worst case scenario'' for the leader trajectory in which the leader vehicle stands still for the entire time horizon. Specifically, we want to show that
    \begin{equation}
    \exists T^{*}\in\R_{>0}:\ \xl(0)-\bx_{1}[\bu](T^{*})-l-d_{\text{safe}}=0\ \wedge\ \dot{\bx}_{1}[\bu](T^{*})=0.\label{eq:minimal_deleceration}
    \end{equation}
    For the proposed control, this translates to
        \[
        \bx_{\circ,1}+\bv_{\circ,1}T^{*}+a_{\min}\tfrac{\big(T^{*}\big)^{2}}{2}=\xl(0)-l-d_{\text{safe}}\ \wedge\ \bv_{\circ,1}+a_{\min}T^{*}=0.
    \]
    Solving this for \((T^{*},a_{\min})\) gives
    \begin{equation}
    a_{\min}\:-\tfrac{\bv_{\circ,1}}{2\big(\xl(0)-\bx_{\circ,1}-l-d_{\text{safe}}\big)},\ T^{*}\:\tfrac{2(\xl(0)-\bx_{\circ,1}-l-d_{\text{safe}})}{\bv_{\circ,1}}.\label{eq:minimal_bound_a_min}
    \end{equation}
    Thus, choosing \(a_{\min}\) as above, the control 
    \[
    \bu_{1}(t)=\begin{cases} a_{\min} & t\in[0,T^{*}]\\
    0 & \text{else}
    \end{cases}
    \]
    satisfies \cref{eq:minimal_deleceration} for all times \(t\in[0,T]\).
\item[\(1\in\Ih\):] 
    Then, there is nothing more to do as by \cref{th:well_posedness_Bando} a solution for \(\bx_{1}\) exists globally and there is no control \(\bu\) applied.
\end{description}
Noting that in both of the previous cases vehicle \(1\) never drives backwards and has a finite deceleration, we can employ the same strategy as in the case \(k=1\) to design a control for the AVs coming afterwards (i.e.,\ \(k\in\Ia, k>1\)). 
Making this uniform across all AVs, we choose
\begin{equation}
a_{\min}=-\max_{k\in \Ia} \tfrac{\bv_{\circ,k}}{2(\bx_{\circ, k-1}-\bx_{\circ, k}-l-d_{\text{safe}})}\label{eq:global_a_min}
\end{equation}
and apply this control to all AVs on the respective time horizons.
This means that we have iteratively constructed a control \(\bu\) satisfying \cref{eq:minimal_deleceration}, which proves the claim. 
The existence and uniqueness of solutions is guaranteed by iteratively applying \cref{th:well_posedness_Bando}. This concludes the proof.
\end{proof}

\begin{remark}[Justification of previously introduced notation]
\label{rem:justification_bx[u]}
Given \cref{cor:well-posedness_system}, we know that by specifying \(\bu\in \bU\) with \(a_{\min}\in\R_{<0}\) sufficiently small, the entire dynamics in \cref{def:dynamics} are fully determined, justifying the notation in \cref{eq:bx[bu]} (which is also used in \cref{def:U}).
\end{remark}

\begin{remark}[The minimal deceleration in \cref{cor:well-posedness_system}]
Note that the \(a_{\min}\) derived in \cref{eq:global_a_min}
might be too conservative in most cases, and it is often possible to choose significantly larger (less negative) values. In numerical tests, this can be identified, and if necessary, the value of \(a_{\min}\) can be adjusted to be smaller. 
\end{remark}

\section{Analysis of the optimal control problem} 
\label{sec:analysis}
Given the problem setup introduced in \cref{sec:setup} and after showing the well-posedness of the system dynamics in \cref{def:dynamics} we show next the existence of a minimizer to the general optimal control problem in \cref{def:ocp}. 
%

\subsection{Existence of a minimizer}
To show the existence of a minimizer of the optimal control problem, we require some stability and continuity results of the solutions with respect to the control. 
Note that the choice of topology is natural as the controls are bounded in \(L^{\infty}\) so that stability of the corresponding solution (second order model) is expected in \(W^{1,\infty}\) by means of compactness.
\begin{lemma}[Continuity of the solution with respect to the control]
\label{lem:stability_solution}
For \(T\in\R_{>0}\) let \((\bu, \tilde{\bu}) \in \bU^2\) with \(\bU \neq \emptyset\) as in \cref{def:U} be given, i.e.\ in particular assume that \(a_{\min}\in\R_{<0}\) as in \cref{ass:control_set} is sufficiently small. 
Then, the following stability estimate holds
\[
\tilde{\bu}\weakstar \bu \text{ in } L^{\infty}\Big((0,T);\R^{|I|}\Big)\ \implies\ \|\bx[\bu]-\bx[\tilde{\bu}]\|_{W^{1,\infty}((0,T);\R^{|I|})}\rightarrow 0
\]
where \(\bx[\bu]\) is defined in \cref{eq:bx[bu]}.
\end{lemma}
\begin{proof}
The proof consists of iteratively going through the one sided coupled system dynamics in \cref{def:dynamics}.
Take \(k=\min_{l\in\Ia} l \), i.e., the index of the first autonomous vehicle. 
Then, this is where we have two different controls \((\bu, \tilde{\bu})\in\bU^2\) applied, so that we obtain for the corresponding dynamics at \(t\in[0,T]\)
\begin{align*}
    \bx_{k}[\bu](t)&=\bx_{\circ,k}+t\bv_{\circ,k}+\int_{0}^{t}\int_{0}^{s} \bu_{k}(z)\dd z\dd s=\bx_{\circ,k}+t\bv_{\circ,k}+\int_{0}^{t} (t-z)\bu_{k}(z)\dd z\\ \bx_{k}[\tilde{\bu}](t)&=\bx_{\circ,k}+t\bv_{\circ,k}+\int_{0}^{t}\int_{0}^{s} \tilde{\bu}_{k}(z)\dd z\dd s=\bx_{\circ,k}+t\bv_{\circ,k}+\int_{0}^{t}(t-z)\tilde{\bu}_{k}(z)\dd z.
\end{align*}
Taking the \(W^{1,\infty}\) norm of the difference, we obtain
\begin{equation}
\begin{aligned}
  \|\bx_{k}[\bu]-\bx_{k}[\tilde{\bu}]\|_{W^{1,\infty}((0,T))}&\leq (1+T)\sup_{t\in[0,T]}\bigg|\int_{0}^{t}\bu_{k}(s)-\tilde{\bu}_{k}(s)\dd s\bigg|\\
  &\quad + \sup_{t\in[0,T]}\bigg|\int_{0}^{t}s\big(\bu_{k}(s)-\tilde{\bu}_{k}(s)\big)\dd s\bigg|.
  \end{aligned}
  \label{eq:42}
\end{equation}
However, if \(\tilde{\bu}\weakstar \bu\), 
the last two terms converge to zero so that we indeed obtain
\begin{equation}
\tilde{\bu}\weakstar\bu \text{ in } L^{\infty}\big((0,T);\R^{|I|}\big)\implies \|\bx_{k}[\bu]-\bx_{k}[\tilde{\bu}]\|_{W^{1,\infty}((0,T))}\rightarrow 0.\label{eq:42_1}
\end{equation}
Continuing with the next vehicle \((k+1)\), it either holds that \((k+1)\in\Ia\) and we just get an identical estimate to the one in \cref{eq:42}, or we have \((k+1)\in\Ih\).
Then, the dynamics for this vehicle are given by the Bando-FtL acceleration from which we obtain for \(t\in[0,T]\) in the integral form for the corresponding velocities
\begin{align}
    &|\bv_{k+1}[\bu](t)-\bv_{k+1}[\tilde{\bu}](t)|\label{eq:estimate_v}\\
    &=\bigg|\alpha \int_{0}^{t}\big(\bv_{k+1}[\tilde{\bu}](s)-\bv_{k+1}[\bu](s)\big)\dd s\notag\\
    &\quad +\alpha\int_{0}^{t}V\big(\bx_{k}[\bu](s)-\bx_{k+1}[\bu](s)-l\big)-V\big(\bx_{k}[\tilde{\bu}](s)-\bx_{k+1}[\tilde{\bu}](s)-l\big)\dd s \notag\\
    &\quad +\beta\int_{0}^{t}\bigg(\tfrac{\bv_{k}[\bu](s)-\bv_{k+1}[\bu](s)}{\big(\bx_{k}[\bu](s)-\bx_{k+1}[\bu](s)-l\big)^{2}}-\tfrac{\bv_{k}[\tilde{\bu}](s)-\bv_{k+1}[\tilde{\bu}](s)}{\big(\bx_{k}[\tilde{\bu}](s)-\bx_{k+1}[\tilde{\bu}](s)-l\big)^{2}}\bigg)\dd s\notag \bigg|
    \intertext{yielding when integrating the last term}
    &\leq \alpha \int_{0}^{t}\big|\bv_{k+1}[\tilde{\bu}](s)-\bv_{k+1}[\bu](s)\big|\dd s\notag\\
    &\quad +\alpha\|V'\|_{L^{\infty}(\R)}\int_{0}^{t}\big|\bx_{k}[\bu](s)-\bx_{k}[\tilde{\bu}](s)\big|+\big|\bx_{k+1}[\bu](s)-\bx_{k+1}[\tilde{\bu}](s)\big|\dd s\notag\\
    &\quad +\beta\Big|\tfrac{1}{\bx_{k}[\bu](t)-\bx_{k+1}[\bu](t)-l}-\tfrac{1}{\bx_{k}[\tilde{\bu}](t)-\bx_{k+1}[\tilde{\bu}](t)-l}\Big|\notag.
\end{align}

Using the lower bound on the headway which is guaranteed in \cref{eq:d_min} and which is invariant with respect to the leader's behavior we have with
\begin{align*}
\bA_{k+1}&\:-\bv_{\circ,k}-\alpha T\|V\|_{L^{\infty}(\R)} +\alpha(\bx_{\circ,k}-\bx_{\circ,k+1}-l)-\beta\tfrac{1}{\bx_{\circ,k}-\bx_{\circ,k+1}-l}\\
\boldsymbol{d}_{\min,k+1}&\: \tfrac{\bA_{k+1}+\sqrt{\bA^{2}_{k+1}+4\alpha\beta}}{2\alpha}
\end{align*}
that
\begin{align}
  \cref{eq:estimate_v}&\leq\alpha \int_{0}^{t}\big|\bv_{k+1}[\tilde{\bu}](s)-\bv_{k+1}[\bu](s)\big|\dd s\label{eq:stability_1}\\
    &\quad +\alpha\|V'\|_{L^{\infty}(\R)}\int_{0}^{t}\!\!\big|\bx_{k}[\bu](s)-\bx_{k}[\tilde{\bu}](s)\big|+\big|\bx_{k+1}[\bu](s)-\bx_{k+1}[\tilde{\bu}](s)\big|\dd s\\
    &\quad +\tfrac{\beta}{\boldsymbol{d}_{\min,k+1}^{2}}\Big(\big|\bx_{k}[\bu](t)-\bx_{k}[\tilde{\bu}](t)\big| +\big|\bx_{k+1}[\bu](t)-\bx_{k+1}[\tilde{\bu}](t)\big|\Big).\label{eq:stability_3}
\end{align}
Combining this for \(t\in[0,T]\) with the trivial estimate
\begin{equation}
|\bx_{k+1}[\bu](t)-\bx_{k+1}[\tilde{\bu}](t)|\leq \int_{0}^{t}\big|\bv_{k+1}[\bu](s)-\bv_{k+1}[\tilde{\bu}](s)\big|\dd s\qquad \forall t\in[0,T]\label{eq:bx[u]_bv[u]}
\end{equation}
we end up with an estimate in the \(W^{1,\infty}\) norm which reads for \(t\in[0,T]\) as
\begin{align*}
    &\|\bx_{k+1}[\bu]-\bx_{k+1}[\bu]\|_{W^{1,\infty}((0,t))}\\
    &=\|\bx_{k+1}[\bu]-\bx_{k+1}[\bu]\|_{L^{\infty}((0,t))}+\|\bv_{k+1}[\bu]-\bv_{k+1}[\bu]\|_{L^{\infty}((0,t))}
    \intertext{taking advantage of \cref{eq:bx[u]_bv[u]} and \crefrange{eq:stability_1}{eq:stability_3}}
    &\leq \Big(1+\alpha+\tfrac{\beta}{\boldsymbol{d}_{\min,k+1}^{2}}\Big)\int_{0}^{t}\|\bv_{k+1}[\bu]-\bv_{k+1}[\tilde{\bu}]\|_{L^{\infty}((0,s))}\dd s\\
    &\quad + \alpha\|V'\|_{L^{\infty}(\R)}\int_{0}^{t}\|\bx_{k+1}[\bu]-\bx_{k+1}[\tilde{\bu}]\|_{L^{\infty}((0,s))}\dd s\\
    &\quad +\Big(\alpha T\|V'\|_{L^{\infty}(\R)}+\beta\tfrac{1}{\boldsymbol{d}_{\min,k+1}^{2}}\Big)\|\bx_{k}[\bu]-\bx_{k}[\tilde{\bu}]\|_{L^{\infty}((0,T))}\\
    &\leq \max\bigg\{\Big(1+\alpha+\tfrac{\beta}{\boldsymbol{d}_{\min,k+1}^{2}}\Big),\alpha\|V'\|_{L^{\infty}(\R)}\bigg\}\int_{0}^{t}\|\bx_{k+1}[\bu]-\bx_{k+1}[\tilde{\bu}]\|_{W^{1,\infty}((0,s))}\dd s\\
    &\quad+\Big(\alpha T\|V'\|_{L^{\infty}(\R)}+\beta\tfrac{1}{\boldsymbol{d}_{\min,k+1}^{2}}\Big)\|\bx_{k}[\bu]-\bx_{k}[\tilde{\bu}]\|_{L^{\infty}((0,T))}.
\end{align*}
Applying Gr\"onwall's inequality \cite[Chapter I, III Gronwall's inequality]{walter}, we obtain
\begin{align*}
    &\|\bx_{k+1}[\bu]-\bx_{k+1}[\bu]\|_{W^{1,\infty}((0,T))}\\
    &\leq \e^{T\max\left\{\left(1+\alpha+\frac{\beta}{\boldsymbol{d}_{\min,k+1}^{2}}\right),\alpha\|V'\|_{L^{\infty}(\R)}\right\}}\\
    &\qquad \cdot \Big(\alpha T\|V'\|_{L^{\infty}(\R)}+\beta\tfrac{1}{\boldsymbol{d}_{\min,k+1}^{2}}\Big)\|\bx_{k}[\bu]-\bx_{k}[\tilde{\bu}]\|_{L^{\infty}((0,T))}.
\end{align*}
As the last term converges to zero as detailed in \cref{eq:42_1}, we  obtain the claimed convergence also for the \((k+1)-\)th vehicle.
Inductively, we obtain the convergence for all \(l\in I\).
\end{proof}

\begin{remark}[The stability estimate in \cref{lem:stability_solution}]\label{rem:stability}
It is worth mentioning that the estimate in \cref{eq:stability_3} relies on the well-posedness of the used car-following model so that a minimal distance between vehicles is guaranteed independently of what the leader does. 
Without this particular estimate, the continuity of the control to state mapping might not hold as the solution would cease to exist (as the vehicles could get arbitrarily close).
Compare in particular \cite[Theorem 3.1]{wang2021optimal} which lacks such an estimate. Further, the authors are not precise in the statement of continuity of the control to state mapping (which can only hold if the car-following model is well-posed) and seem to not work with the weak-star convergence of the control set. However, this becomes problematic when later claiming that the set of admissible controls is compact in the proof of \cite[Theorem 3.1]{wang2021optimal}. Clearly, when prescribing box constraints, as we also do in this contribution, the corresponding set is weak star and as such weakly compact, but clearly not compact in the \(L^{p},\ p\in[1,\infty)\) topology, and for sure not in the uniform topology. This could be fixed by postulating that the set of admissible controls is uniformly \(TV\) bounded or uniform Lipschitz (compare Ascoli-Arzela) which would imply this compactness in \(L^{p},\ p\in[1,\infty),\ L^{\infty}\) respectively, however, it would restrict the control set unnecessarily and would be more difficult to implement in the optimization algorithm. We have demonstrated in the former continuity estimate that it suffices for  \(\bu\) to converge in the weak star topology. 

This being said, it becomes clear why their proof of existence of minimizers becomes trivial as the convergence in the norm (which they claimed falsely) is compatible with the chosen objective functions. Anticipating \cref{th:existence_minimizer}, we will give the existence of a minimizer when the set of admissible controls \(\bU\) is not further restricted, and only the convexity of the objective function with regard to the control is assumed (see \cref{def:ocp}). Such could also be weakened to lower weakly semi-continuity, but we will not detail this further.
\end{remark}

The following theorem establishes the existence of a minimizer by following the classical approach of taking a minimizing sequence and showing that it is compact in the proper topology, so that one can pass to the limit in the objective function, dynamics, and constraints.
\begin{theorem}[Existence of a minimizer]
\label{th:existence_minimizer}
Let the optimal control problem in \cref{def:ocp} be given together with the therein stated assumptions. Then, for a sufficiently small \(a_{\min}\in\R_{<0},\) there exists \(\bu^{*}\in\bU\) so that it holds
\begin{equation*}
\begin{split}    J\Big(\bx[\bu^{*}],\bv[\bu^{*}],\bu^{*}\Big)&=\inf_{\bu\in\bU}\int_{0}^{T}L(\bx[\bu](t),\bv[\bu](t),\bu(t))\dd t+ J_{T}(\bx[\bu](T),\bv[\bu](T))
\end{split}
\end{equation*}
with \(\bx[\bu],\bv[\bu]\) as defined in \cref{eq:bx[bu]}. 
\end{theorem}
\begin{proof}
As the objective function is bounded from below we know that there exists an infimum and a sequence \(\big(\bu_{n}\big)_{n\in\N_{\geq1}}\subset \bU\) such that
\[
\lim_{n\rightarrow\infty} J(\bx[\bu_{n}],\bv[\bu_{n}],\bu_{n})= \inf_{\bu\in\bU}J(\bx[\bu],\bv[\bu],\bu).
\]
Due to the assumptions on \(\bu_{n}\in\bU \subset \bU_{a_{\min}}^{a_{\max}}\) (see \cref{ass:control_set}) we have
\[
\|\bu_{n}\|_{L^{\infty}((0,T);\R^{|I|})}\leq \max\{a_{\max},-a_{\min}\}
\]
and thus 
\[\sup_{n\in\N_{\geq1}}\|\bu_{n}\|_{L^{\infty}((0,T);\R^{|I|})}\leq \max\{a_{\max},-a_{\min}\},\]
so that we can invoke Banach–Alaoglu–Bourbak theorem \cite[Theorem 3.16]{brezis} to conclude that there exists \(\bu^{*}\in\bU\) with 
\[
a_{\min}\leq \bu^{*}(t)\leq a_{\max},\ t\in[0,T] \text{ a.e.}
\]
and a subsequence (which we again index by \(n\in\N_{\geq1}\)) with
\[
\bu_{n}\xrightharpoonup[n\rightarrow\infty]{*}\bu^{*} \text{ in } L^{\infty}\big((0,T);\R^{|I|}\big).
\]
As it holds that \(\forall n\in\N_{\geq1}\)
\[
\bv_{\circ}+\int_{0}^{t}\bu_{n}(s)\dd s\geq0\quad \forall t\in[0,T]
\]
we can take \(n\rightarrow\infty\) and, using the weak-star convergence of \(\bu_{n}\rightarrow \bu^{*},\) we obtain
\[
\bv_{0}+\int_{0}^{t}\bu^{*}(s)\dd s\geq 0\quad \forall t\in[0,T]
\]
which implies that \(\bu^{*}\in \bU_{a_{\min}}^{a_{\max}}\).
Next, we show that \(\bu^{*}\in \bU\) which means that \(\bU\) is closed with respect to the weak star topology. Recalling the definition of \(\bU\) we have
\begin{equation}
\bx_{i-1}[\bu_{n}](t)-\bx_{i}[\bu_{n}](t)-l\geq d_{\text{safe}}\quad \forall t\in[0,T]\ \forall i\in \Ia\ \forall n\in\N_{\geq1}.\label{eq:123123}
\end{equation}
Using the result in \cref{lem:stability_solution}, we know that
\[
\lim_{n\rightarrow\infty}\|\bx[\bu_{n}]-\bx[\bu^{*}]\|_{W^{1,\infty}((0,T);\R^{|I|})}=0,
\]
i.e., that we can pass to the limit in \cref{eq:123123} to obtain
\[
\bx_{i-1}[\bu^{*}](t)-\bx_{i}[\bu^{*}](t)-l\geq d_{\text{safe}}\ \forall t\in[0,T]\ \forall i\in \Ia,
\]
which indeed means that \(\bu^{*}\in \bU\). By the very construction we also have that \(\bx[\bu^{*}],\bv[\bu^{*}]\) are solutions of the system dynamics in \cref{def:dynamics}. 
So it remains to show that \(\bu^{*}\) is a minimizer. To this end, write
\begin{align*}
J(\bx[\bu^{*}],\bv[\bu^{*}],\bu^{*})&=\lim_{n\rightarrow \infty} J(\bx[\bu_{n}],\bv[\bu_{n}],\bu^{*})\leq \lim_{n\rightarrow\infty} J(\bx[\bu_{n}],\bv[\bu_{n}],\bu_{n})\\ &=\inf_{\bu\in\bU}J(\bx[\bu],\bv[\bu],\bu)
\end{align*}
where the first equality is a consequence of the uniform convergence of \(\bx[\bu_{n}]\) and \(\bv[\bu_{n}]\) in \(W^{1,\infty}\big((0,T);\R^{|I|}\big)\), 
while the inequality is a consequence of the lower weakly semi-continuity of the objective function which is guaranteed by the convexity assumption in \cref{def:ocp}. Altogether, we have constructed a minimizer \(\bu^{*}\in \bU,\) proving the claim.
\end{proof}

\begin{remark}[Minimal assumptions on the control set and objective function]
In \cref{th:existence_minimizer}, we give the existence of a minimizer when the set of admissible controls \(\bU\) is not further restricted, and only the convexity of the objective function with regard to the control is assumed (see \cref{def:ocp}). Such could also be weakened to lower weakly semi-continuity, but we will not detail this further.
\end{remark}

\begin{remark}[Generalization to different car following models]
The proposed optimal control framework is applicable to other second order car-following models. 
%
However, the well-posedness of the model needs to be satisfied independently of the proposed controls.
\end{remark}

\begin{remark}[Formulation with additive acceleration for the autonomous cars]
Making the assumption that car-following behavior (i.e.\ human behaviour when using the nomenclature of this contribution) should be close to what an autonomous vehicle might aim for, one could define the control in \cref{def:ocp} and in the dynamics \cref{eq:dynamics} additively, meaning that replacing the double integrator for \((t,i)\in[0,T]\times\Ia\)
\[
\dot{\bv}_{i}(t)=\bu_{i}(t)\quad  \text{by} \quad \dot{\bv}_{i}(t)=\Acc\big(\bx_{i}(t),\bx_{i-1}(t),\bv_{i}(t),\bv_{i-1}(t)\big) +\bu_{i}(t) .
\]
However, one might run into a well-posedness problem when allowing \(\bu_{i}\) to be positive. 
Since the goal is to smooth traffic, it might be sufficient to allow only deceleration in order to create reasonable gaps between vehicles. We do not go into details, but refer to \cite{wang2021optimal} where such an approach was taken.
\end{remark}

\section{Numerical Realization} 
\label{sec:method}
In this section, we describe a numerical approach for solving the optimal control problem in \cref{def:ocp}. 
We will use a direct method, which essentially discretizes the optimal control problem and transforms it into a finite-dimensional non-linear program then solves the discrete problem by means of gradient descent. 
We use a penalty approach to handle the non-linear state constraints described in the admissible control set $\bU$ \cref{eq:U}. 
We employ the adjoint method to derive an analytical expression for the continuous-time gradient of the problem. 
To discretize the problem, we approximating the control using piece-wise constant basis functions.
We numerically integrate the system dynamics \cref{eq:dynamics} (forward integration) and the adjoint system (backward integration) using standards numerical integration schemes, while ensuring that the methods for forward and backward integration are compatible in the sense that they are performed on the same fixed grid and using the same integration method. 
Finally, the discrete problem can then be solved using standard gradient-based optimization techniques.

\subsection{Gradient Evaluation} 
\label{sec:gradient_computation}
To simplify the notation, we introduce the combined state vector for \(t\in[0,T]\)
\begin{align}
    \by(t) = \begin{bmatrix}
        \bx(t) \\
        \bv(t)
    \end{bmatrix}, 
\end{align}
and the vector valued function $\bbf: \R^{2|I|} \mapsto \R^{2|I|}$ that describes the system dynamics in \cref{def:dynamics} such that
\begin{align}
    \dot{\by}(t) = \bbf(t, \by(t), \bu(t)),\ t\in[0,T].
\end{align}

The aim is to compute the variation of the objective function with respect to the control
\begin{equation}
    \begin{aligned}
        J_{\bu}(\by, \bu)[\tilde{\bu}] = &  \int_{0}^{T} L_{\by}(\by(t), \bu(t)) \by_{\bu}(t) \tilde{\bu}(t) + L_{\bu}(\by(t), \bu(t)) \dd t \\
    & + S_{\by}(\by(T)) \by_{\bu}(T) \tilde{\bu}(t), 
    \end{aligned}
\end{equation}
where $t\mapsto\tilde{\bu}(t)$ is a direction function. 
Here, we use the notation $ J_{\bu} = \frac{\partial J}{\partial \bu}$. 
Evaluating the expressions $\by_{\bu}$ is difficult due to the implicit dependence of the combined state vector on the control. 
For this, we use the adjoint method to evaluate $J_{\bu}$ without the need to explicitly evaluating $\by_{\bu}$. 
This approach has been extensively described in the literature \cite{chavent2010nonlinear, cao2003adjoint, givoli2021tutorial} and widely used to solve optimal control problems in various application domains \cite{pikulinski2021adjoint, bayen2004adjoint, lin2014control, eichmeir2021adjoint, gugat2005optimal}. Below, we give the derivation for the adjoint system associated with our optimal control problem. 

As a first step, we write the Hamiltonian
\begin{align}
\label{eq:hamiltonian}
    H(\by, \bu, \bzeta) = \int_{0}^T L(\by(t), \bu(t)) + \bzeta^{\top}(t) \big(\dot{\by}(t) - \bbf (t, \by(t), \bu(t))\big) \dd t + S(\by(T)), 
\end{align}
where $\bzeta \in W^{1,\infty}\big((0,T);\R^{2|I|})$ is a vector of Lagrange multipliers. 
We then compute the first variation of $H$ with respect to \(\by\) and \(\bu\)
\begin{equation}
\begin{aligned}
    \delta H(\by, \bu, \bzeta) = & \int_0^T \big( L^{\top}_{\by}(\by(t), \bu(t)) \delta\by(t) + L^{\top}_{\bu}(\by(t), \bu(t)) \delta \bu(t) \big) \dd t \\
    & + \int_0^T \bzeta^{\top}(t) \big(\delta \dot{\by}(t)  + \bbf_{\by}(t, \by(t), \bu(t)) \delta \by(t)+ \bbf_{\bu}(t, \by(t), \bu(t)) \delta \bu(t) \big) \dd t \\
    & + S^{\top}_{\by}(\by(T)) \delta \by(T).
\end{aligned}
\label{eq:hamiltonian_var}
\end{equation}
%
Using integration by parts, we write 
\begin{align}
\label{eq:integration_byparts}
    \int_{0}^T \bzeta^{\top}(t) \delta\dot{\by}(t)\dd t & =  \bzeta^{\top}(T) \delta \by(T)
    - \bzeta^{\top}(0) \delta \by(0) -\int_0^T \dot{\bzeta}^{\top}(t) \delta\by(t) \dd t
\end{align}
By plugging \cref{eq:integration_byparts} in \cref{eq:hamiltonian_var}
and rearranging the terms we obtain
\begin{align*}
    \delta H(\by, \bu, \bzeta) = & \int_0^T \big( L^{\top}_{\bu}(\by(t), \bu(t)) - \bzeta^{\top}(t)\bbf_{\bu}(t, \by(t), \bu(t)) \big) \delta \bu(t) \dd t \\
     & + \int_0^T \big( L^{\top}_{\by}(\by(t), \bu(t)) - \bzeta^{\top}(t)\bbf_{\by}(t, \by(t), \bu(t)) + \dot{\bzeta}^{\top}(t) \big)\delta \by(t) \dd t \\
     & + \bzeta^{\top}(0) \delta \by(0)\\
     & + \big(S^{\top}_{\by}(\by(T)) -  \bzeta^{\top}(T)\big) \delta \by(T). 
\end{align*}
Note that $\delta \by(0) = 0 $ because the initial conditions are fixed. With the Lagrange multipliers satisfying the final value problem (adjoint system) 
\begin{equation}
    \begin{aligned}
    \bzeta(T) & = S_{\by}(\by(T))  \\
    \dot{\bzeta}(t)  & = -L_{\by}(\by(t), \bu(t)) + \bbf^{\top}_{\by}(t, \by(t), \bu(t)) \bzeta(t), 
\end{aligned}
\label{eq:adjoint_system}
\end{equation}
the expression for the gradient reduces to 
\begin{align}
     \delta H(\by, \bu, \bzeta)  = & \int_0^T \big( L^{\top}_{\bu}(\by(t), \bu(t)) - \bzeta^{\top}(t)\bbf_{\bu}(t, \by(t), \bu(t)) \big) \delta \bu(t) \dd t. 
\label{eq:simplified_var}
\end{align}
The linear system of ODEs (with respect to \(\bzeta\)) in \cref{eq:adjoint_system} describes a system of final value problems. The conditions at the end time $T$ are defined by the terminal cost function $S(\by(T))$. Thus, as commonly done, the adjoint system is integrated in the backward direction over the interval $[0, T]$. 


We now explicitly write the adjoint system in equation \cref{eq:adjoint_system} for our specific dynamics in \cref{eq:dynamics}.
%
%
The dynamics are only coupled in one direction, i.e., the state of a vehicle $i \in I$ depends only on its preceding vehicle $i-1$. 
This makes the time-dependent Jacobian matrix $t\mapsto\bbf_{\by}(t, \by(t), \bu(t))$ have a sparse structure of the following form. For \(t \in [0,T]\), 
\begin{align}
    \bbf_{\by}(t, \by(t), \bu(t)) = \begin{bmatrix} \boldsymbol{A}(t) & \boldsymbol{B}(t) \\ \boldsymbol{C}(t)& \boldsymbol{D}(t)
    \end{bmatrix}
\end{align}
where $\boldsymbol{A}(t), \boldsymbol{B}(t), \boldsymbol{C}(t), \boldsymbol{D}(t) \in \R^{|I|\times |I|}$, \ $\boldsymbol{A}
(t) = \boldsymbol{0}$, $\boldsymbol{B}(t) = \boldsymbol{I}$, and $\boldsymbol{C}(t)$ and $\boldsymbol{D}(t)$ are lower-bidiagonal matrices. 
The entries of $\boldsymbol{C}$ (omitting the time and function dependencies) are 
\begin{align}
    \boldsymbol{C}_{i, i} & = \begin{cases}
        \Acc_{\bx_{i}}(\bx_i, \bx_{i-1}, \bv_{i}, \bv_{i-1}) &\! \forall i \in I_h\\
        0 & \!\forall i \in I_a, 
    \end{cases}\\
    \boldsymbol{C}_{i+1, i} & = \begin{cases}
        \Acc_{\bx_{i}}(\bx_{i+1}, \bx_{i}, \bv_{i+1}, \bv_{i}) & \!\forall i \in I_h \setminus \{M+N\}\\
        0 & \!\forall i \in I_a \setminus \{M+N\}. 
    \end{cases}
\end{align}
The entries of $\boldsymbol{D}$ are
\begin{align}
    \boldsymbol{D}_{i, i} & = \begin{cases}
        \Acc_{\bv_{i}}(\bx_i, \bv_{i-1}, \bv_{i}, \bv_{i-1}) &\!\! \forall i \in I_h\\
        0 &\!\! \forall i \in I_a, 
    \end{cases}\\
    \boldsymbol{D}_{i+1, i} & = \begin{cases}
        \Acc_{\bv_{i}}(\bx_{i+1}, \bx_{i}, \bv_{i+1}, \bv_{i}) &\!\! \forall i \in I_h \setminus \{M+N\}\\
        0 &\!\! \forall i \in I_a \setminus \{M+N\}. 
    \end{cases}
\end{align}
Recall that $\bx_0(t) = \xl(t)$ and $\bv_0(t) = \vl(t)$. 
The Jacobian matrix $\bbf_{\bu}(t, \by(t), \bu(t))$ has the form 
\begin{align}
    \bbf_{\bu}(t, \by(t), \bu(t)) = 
    \begin{bmatrix}
    \mathbf{0} \\ \boldsymbol{E}
    \end{bmatrix}
\end{align}
where $\boldsymbol{E} \in \mathbb{R}^{|I|}$ is a diagonal matrix with entries $E_{i, i} = 1$ if $i \in I_a$ and zero otherwise. 
This yields the following representation of the gradient $t\mapsto\bG(\by(t), \bu(t), \bzeta(t)) \in L^{\infty}\big((0,T);\R^{|I|})$. 
For \( t \in [0, T]\), 
\begin{align}
\label{eq:grad}
\bG_i(\by(t), \bu(t), \bzeta(t)) \: 
\begin{cases}
      L_{\bu_i}(\by(t), \bu(t)) - \bzeta_{i + |I|}(t) & i \in I_a\\  
      0 & i \in I_h.
\end{cases}
\end{align}


\subsection{State Constraints}
\label{sec:constraints} 
The formulated optimal control problem has the following state constraints encoded in the admissible control set $\bU$
\begin{align}
    \bx_{i-1}(t) - \bx_i(t) - l & \geq d_{\text{safe}} & \forall i \in I_a, \ \forall t \in [0, T] \label{eq:headway_const}\\ 
    \bv_{i}(t) & \geq 0 & \forall i\in I_a,\ \forall t \in [0, T] \label{eq:vel_cons}, 
\end{align}
which need to be satisfied at every time step $t\in [0, T]$. 
The constraints \cref{eq:vel_cons} which ensure the non-negativity of the autonomous vehicles' velocities are linear in the control and can be written as
\begin{align}
\label{eq:velocity_const}
    v_{o,i} + \int_{0}^{t} u(s) \dd s \geq 0 & & \forall i\in I_a,\ \forall t \in [0, T]. 
\end{align}
These constraints will be discretized and handled by the optimization solver. 
The other set of constraints are in general non-linear and typical optimization solvers can not handle them efficiently. For this we employ an exact penalty method to enforce these constraints. Define the following transformation 
\begin{align}
    h_i(t, \by(t) ) = \min{\{ \bx_{i-1}(t) - \bx_i(t) - l - d_{\text{safe}}, 0\}}^2 && \forall i\in I_a,\ \forall t \in [0, T], 
\end{align}
and note that the function $h_i(t, \by(t))$ is continuously differentiable.
The constraints in \cref{eq:headway_const} can then be equivalently written as: 
\begin{align}
    \int_{0}^{T} h_i(t, \by(t)) \dd t = 0 && \forall i \in I_a. 
\end{align}
This converts an infinite number of constraints to a single constraint for each autonomous vehicle in the system. 
We then augment these constraints to the objective function $J(\by, \bu)$. For a penalty parameter $\mu \in \R_{+}$ we write the augmented objective function 
\begin{align}
    \tilde{J}(\by, \bu) = \int_{0}^{T} L(\by, \bu) \dd t + S(\by(T)) + \mu \sum_{i \in I_a} \int_{0}^{T} h_i(t, \by(t)) \dd t. 
\end{align}
%
%
In our numerical implementation, we iteratively adjust the penalty parameter $\mu$ by gradually increasing its value until the constraint violations are within an acceptable tolerance. 


\subsection{Discretization Scheme} 
\label{sec:disc_scheme}
We discretize the continues-time optimal control problem through the parametrization of the control function and numerical integration of the system states and adjoint variables. 
We briefly describe this discretization scheme in the following. 
We refer to \cite{gerdts2023optimal} for possible alternative discretization methods. 

\subsubsection{Control Parametrization}
We partition the time horizon $[0, T]$ with the grid 
\begin{align}
    \mathcal{T} = \{\tau_0, \tau_1, \dots, \tau_p:\ p\in \R_{+},\ \tau_0 = 0,\ \tau_p = T,\ \tau_{k-1} < \tau_k, \ \forall k =1, \dots p \}
\end{align}
The control is then approximated by a piece-wise constant function 
\begin{align}
\label{eq:pwc_control}
    \hat{\bu}(t;\bomega) = \sum_{k = 1}^{p} \bomega^{k} \chi_{[\tau_{k-1}, \tau_k)}(t), 
\end{align}
where $\bomega^{k} \in \R^{|I|}$ is the parameters vector corresponding to the $k^{th}$ interval and 
\begin{align*}
    \chi_{Z}(t) = \begin{cases}
        1 & \text{if } t \in Z\\
        0 & \text{otherwise},  
    \end{cases}
\end{align*}
This piece-wise constant parametrization is a common choice in the literature due to its simplicity and convergence properties \cite{lin2014control}. 
However, one could consider piece-wise linear (and continuous) parametrization or other more regular basis functions. 

Let $\bomega = \left[\bomega^1, \bomega^2, \dots, \bomega^{p}\right] \in \R^{p|I|}$ be the combined parameters vector for all intervals. These are the optimization variables of the discrete problem. 
We define the feasible parameter set 
\begin{align}
    \hat{\bU}_{\mathcal{T}} = \{\bomega \in \R^{p|I|} : \bomega^{k} \in \hat{\bU},\ k = 1, 2, \dots, p\}. 
\end{align}
where
\begin{align}
    \hat{\bU} = \{\bomega^k \in \R^{|I|} : \bomega_i^{k} \equiv 0 \ \forall i \in I_h \quad a_{\min}\leq \bomega_i^k \leq a_{\max} \ \forall i \in I_a\}. 
\end{align}

The linear velocity inequality constraints in \eqref{eq:velocity_const} are evaluated by approximating the system states. We use a third-order explicit \textit{Runge-Kutta} (RK3) integration scheme with the fixed mesh \(\mathcal{T}\) to approximate the states and evaluate these constraints.  

\subsubsection{State and Adjoint Variables Integration} Using the piece-wise constant control in \eqref{eq:pwc_control}, we write the system dynamics as 
\begin{align}
    \dot{\by}(t) = \bbf(t, \by(t), \bomega^{k}) & & \forall t \in [\tau_{k-1}, \tau_k), \ \forall k=1, 2, \dots, p, 
\label{eq:approximate_dynamics}
\end{align}
and denote by $\y[\bomega]$ the solution of the dynamics corresponding to this parametrization.
We similarly write the adjoint system as
\begin{align}
    \dot{\bzeta}(t)  & = -L_{\by}(\by(t), \bomega^{k}) + \bbf^{\top}_{\by}(t, \by(t), \bomega^{k}) \bzeta(t) & & \forall t \in [\tau_{k-1}, \tau_k), \ \forall k=1, 2, \dots, p, 
\label{eq:approximate_adjoint}
\end{align}
and denote by $\bzeta[\bomega]$ the solution to the adjoint system. 

The solutions $\y[\bomega]$ and $\bzeta[\bomega]$ are approximated using an RK3 integration scheme on a fixed grid. 
The use of the fixed grid, rather than the common adaptive grid used for RK schemes, is to ensure that both systems of ODEs are solved on the same time discretization. 
The use of adaptive mesh size creates fundamental difference in the numerical scheme which is not considered in this study. 
Further, a mesh which can potentially change in each gradient evaluation can also cause issues in the optimization routine as we would be optimizing over the discretization noise which changes each iteration.
The fixed grid on which the states and adjoint variables are evaluated
\begin{align}
    \mathcal{S} = \{t_{0}, t_1, \dots, t_r:\ r\in \R_{+},\ t_0 = 0,\ t_p = T,\ t_{i-1} < t_i, \ \forall j =1, \dots r \} 
\end{align}
is selected as a refinement of the control grid (i.e. \(r >> p\) and \(\mathcal{T} \subset \mathcal{S})\). 
This is done to avoid numerical instability and ensure that the objective function is sensitive to every control parameter in the vector $\bomega$.
%

Using the piece-wise control parametrization and the approximate state and adjoint variables yields the following approximation of the objective function 
\begin{align}
\label{eq:approx_objective}
    \hat{J}(\bomega)& =  \sum_{k = 1}^{p}\int_{\tau_{k-1}}^{\tau_k} L\big(\y[\bomega](t),\bomega^{k}\big)\dd t + S\big(\y[\bomega](T)\big). 
\end{align}
And the following approximation of the gradient of the objective function with respect to the parameters vector $\bomega$ 
\begin{align}
\label{eq:approx_grad}
    \hat{\bG}(\bomega) = \int_{\tau_{k-1}}^{\tau_k} \bG(\by[\bomega](t), \bomega^{k}, \bzeta[\bomega](t)) \dd t && \forall k = 1, 2, \dots, p. 
\end{align}
The integrals are approximated using a trapezoidal integration scheme on the grid $\mathcal{S}$. 

\begin{remark}[Choice of the integration scheme]
    The RK3 scheme is a suitable choice given the piece-wise constant control parametrization as the discontinuity in the control does not affect the quality of the state integration. To see this, we note that RK3 evaluates the controller at points $\bu^{p}(t| \bsigma), \bu^{p}(t + 0.5t | \bsigma)$ and $\bu^{p}(t + 0.75t| \bsigma)$ which are all the same to under our control parameterization. This also ensures that the states approximated on the grid $\mathcal{T}$ used to evaluate the constraints, and the states approximated on the finer grid $\mathcal{S}$ used to evaluate the objective function and its gradient are comparable. 
\end{remark}

\subsection{Optimization Solver}
The discretization scheme discussed in section \cref{sec:disc_scheme} defines a finite optimization problem with $p \times |I_a|$ effective variables and $p \times |I_a|$ linear inequality constraints, and $2\times |I_a|$ box constraints on the control variables. 
In our implementation, we use the sequential quadratic programming (SQP) implementation in MATLAB's \textit{fmincon} solver \cite{MATLAB}. The problem we formulate is in general non-convex, thus any gradient-based solver we use is only guaranteed to converge locally. This makes the initialization of the algorithm crucial and we will discuss this in further details in \cref{sec:experiemnts}. 
%
We summarize the proposed numerical approach in \cref{alg}. 
 
\begin{algorithm}[H] 
   \caption{Optimal Control Solver Using Adjoint Gradients}
\label{alg}
 \KwData{Lead vehicle trajectory $\xl(t)$; Initial guess $\bomega^{(0)}$; Initial penalization parameter $\mu^{(0)}$; Initial conditions of the platoon $\bx_{\circ}$, $\bv_{\circ}$}
 \KwResult{Locally optimal control parameters $\bomega^{\star}$}
 \While{constraints are not satisfied}{
  \For{\(i = 1:N\)}{\;\\
    Run a forward simulation of the system dynamics in \cref{eq:approximate_dynamics} using RK-3\;\\
    Evaluate the objective function as in \cref{eq:approx_objective}\;\\
    Run a backward simulation of the adjoint system in \cref{eq:approximate_adjoint} using RK-3\;\\
    Evaluate the gradient function as in \cref{eq:approx_grad}\;\\
    Take gradient step and update $\bomega^{(i)}$\;
  }\;
  Set $\bomega^{(0)} = \bomega^{(N)}$ \;\\
   Update $\mu$\;
 }
\end{algorithm}
\section{Numerical Experiments} 
\label{sec:experiemnts}
In this section we instantiate the optimal control problem in \cref{def:ocp} with a set of concrete experiments. 
%
We also numerically highlight the capabilities and limitations of the autonomous vehicles' ability to regulating traffic.

\subsection{Problem Instant}
We consider an instance of the optimal control problem with the goal of \textit{smoothing stop-and-go waves} and \textit{reducing the energy consumption} of a mixed-autonomy platoon of vehicles following a pre-specified lead vehicle trajectory. 
We use the $L^2$-norm of the acceleration of the vehicles as a simple proxy for these two objectives. 
Precisely, we consider the objective function
\begin{align}
\label{eq:acc_objective}
    J(\bx, \bv, \bu) = \sum_{i\in I_a} \int_{0}^{T} \bu_i^2(t)\dd t + \sum_{i \in I_h} \int_{0}^{T}\Acc^2\big(\bx_{i-1}(t), \bx_i(t), \bv_{i-1}(t), \bv_{i}(t)\big) \dd t, 
\end{align}
where the terminal cost is set to $S(\bx(T), \bv(T)) = 0$. 

\begin{remark}[Fuel consumption model]
\label{remark:energy_function}
The literature is rich with energy models with varying levels of fidelity aimed at computing the fuel consumption of a given trajectory. 
We consider one such model proposed in \cite{lee2021integrated} which can be used to evaluate the fuel consumption of a trajectory by integrating an instantaneous energy function $E(v, u)$ over the time horizon $[0, T]$. 
This function can be partially integrated and takes the form 
\begin{equation}
\begin{aligned}
\int_{0}^{T}E(\bv_{i}(t),\bu_{i}(t))\dd t & =C_{0}+ C_{1}\bx_{i}(T)+p_{0}\bv_{i}(T) \\
&\quad +\tfrac{p_{1}}{2}\bv(T)^{2} +\tfrac{p_{2}}{3}\bv_{i}(T)^{3}+\sum_{k=2}^{3}C_{k}\int_{0}^{T}\bv_{i}(t)^{k}\dd t\\
&\quad +\int_{0}^{T}\!\!\! q_{0}\max(\bu_{i}(t),0)^{2}+q_{1}\max(\bu_{i}(t),0)^{2}\bv_{i}(t)\dd t,
\end{aligned}
\label{eq:energy_integrated_out}
\end{equation}
where \(i \in I\) and the parameters $C_j, p_j$, and $q_k$ are calibrated using experimental data with \(j \in \{1, 2, 3\}\) and \(k \in \{1, 2\}\).
We note that this energy function does not incur a cost on negative acceleration. 
This can be exploited if we optimize this function and results in a solution trajectory in which the controlled acceleration drops heavily at the end of the time horizon as all quantities in the objective functions will then decrease as well (see \cref{eq:energy_integrated_out}).
This behavior is an artifact of the energy model and the fact that we optimize over a fixed time horizon. 
For this reason, we refrain from using the model as an objective function and only use it for evaluation.
\end{remark}

Optimizing the $L^2$ norm of the acceleration alone (or even the energy function) can result in the AV driving at a slower speed than its leader. 
This behavior can negatively affect other metrics of the platoon such as the throughput, density, and average velocity. 
There are many possible design choices to remedy this effect and pull the AV to exit the road. 
Here, we do so by imposing additional constraints on the trajectory of the AV to regulate its headway and ensure that it does not fall far behind its leader. 
We introduce the following additional set of constraints 
\begin{align}
\label{eq:max_headway}
    \bx_{i-1}(t) - \bx_i(t) - l & \leq d_{\max} & \forall i \in I_a, \ \forall t \in [0, T]
\end{align}
for parameter $d_{\max} \in \R_{+}$. 
We handle these constraints through the same penalization approach used for the minimum headway constraints described in \cref{sec:constraints}. 

\begin{remark}[Well-posedness under maximum headway constraint]
We note that the well-posedness results shown in \cref{cor:well-posedness_system} do not take into consideration these maximum headway constraints introduced in \cref{eq:max_headway}. 
These results can be extended for these constraints. 
However, here we will limit our discussion to the numerical solution of the problem under such constraints. 
\end{remark}

In all experiments, we do not impose the box constraints on the control as the objective function favors small acceleration. 
We will demonstrate through numerical examples that the solution remains bounded and is proportional to the acceleration of the leader.

Based on the above description we define the following instance of the problem
\begin{equation}
\begin{aligned}
\inf_{\bu\in \R^{|I|}} & \quad  \int_{0}^{T}\!\!\!\! \sum_{i \in I_a}\! \bu_i^2(t)\!+\!\! \sum_{i \in I_h} \! \big(\Acc^2(\bx_{i-1}(t), \bx_i(t), \bv_{i-1}(t), \bv_{i}(t))\big) \dd t\\
     \text{ where } & \quad (\bx, \bv) \text{ satisfies } \cref{eq:dynamics} \text{ and } \forall i \in I_a,\ \forall t\in[0,T] \\ 
    & \quad \bx_{i-1}(t) - \bx_i(t) - l \geq d_{\text{safe}}\\
     & \quad \bx_{i-1}(t) - \bx_i(t) - l \leq d_{\max}\\
     & \quad  \bv_i(t) \geq 0.
\end{aligned}
\label{eq:ocp_instance}
\end{equation}

\subsection{Experiments} We perform three sets of experiments to investigate the effect of introducing autonomous vehicles with different platoon configurations and with variants of the objective function in \ref{eq:acc_objective}.

\paragraph{Parameters Setup}
The problem constraints are parameterized by the safety distance in \cref{def:U} and the maximum allowable headway in \cref{eq:max_headway}. 
We set the safety distance $d_{\text{safe}} = 5 m$ and the maximum headway $d_{\max} = 120 m$. 
To discretize the continuous time problem we consider equidistant grids for both the control and state variables (see \cref{sec:disc_scheme}). 
The control mesh size is selected to be $5$ second while the states mesh size is $0.1$ seconds. 

\paragraph{Leader's Trajectory} 
All experiments in this section are performed with an empirical lead vehicle trajectory collected on a segment of the I-24 highway in Nashville, TN \cite{osti_1855608}. 
This trajectory (shown in \cref{fig:empirical_trajectory}) resembles a congestion scenario with stop-and-go traffic patterns.  
%
%
\begin{figure}[ht]
    \centering
    \begin{subfigure}[b]{0.46\textwidth}
        \centering
        \includegraphics[width=\textwidth]{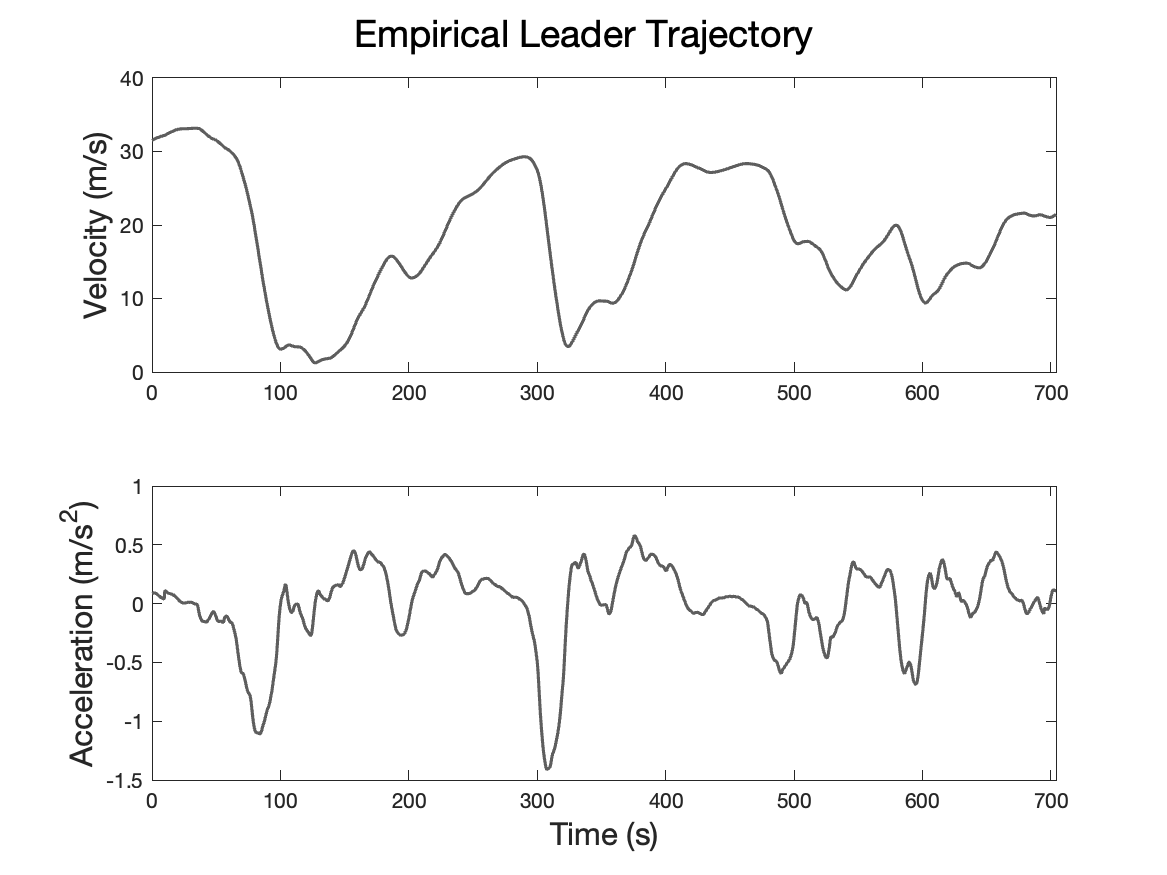}
        \caption{Leader's trajectory}
        \label{fig:empirical_trajectory}
    \end{subfigure}
    ~
    \begin{subfigure}[b]{0.46\textwidth}
        \centering
        \includegraphics[width=\textwidth]{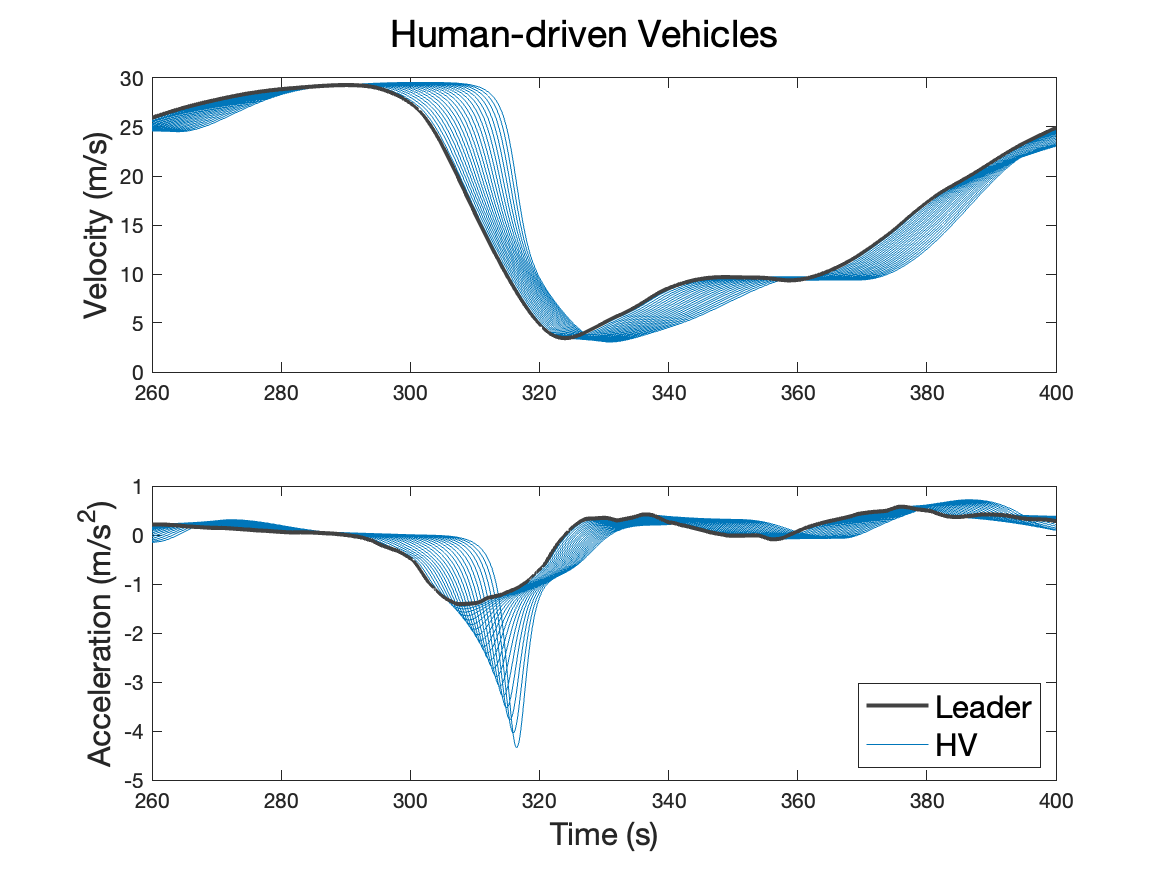}
        \caption{Trajectories of Bando-FtL model}
        \label{fig:hv_empirical}
    \end{subfigure}
    \caption{(a) Velocity and acceleration profiles of the considered lead vehicle. (b) Segments showing velocity and acceleration of a platoon of 20 human-driven vehicles following the leader's trajectory in (a). These segments highlight the instability in the behavior of the Bando-FtL car following model used in this study.}
    \label{fig:leader}
\end{figure}

\paragraph{Human-driven Vehicles}
In all experiments, the human-drive vehicles' (HVs) acceleration is modeled by the Bando-FtL model described in \cref{defi:bando_ftl} with parameters $\alpha = 0.1$ and $\beta = 525$.
In \cref{fig:hv_empirical}, we illustrate the behavior of a full human-driven platoon consisting of 20 vehicles following the considered lead vehicle's trajectory. 
Note that the model propagates and amplifies the stop-and-go waves -- a behavior that indicates string instability.
Such behavior is also observed in commercially deployed adaptive cruise control systems \cite{gunter2020commercially}. It is also believed that human driving behavior exhibits such instability. 
In the experiments we present here we aim to investigate the effect of introducing a various number of AVs on minimizing such instability and dampening stop and go waves in this platoon. 
%

\paragraph{Evaluation Metrics} 
We use two metrics to evaluate the performance of a platoon. 
The first is the total acceleration, evaluated as the sum of the \(L^2\) norm of the acceleration of all vehicles in the platoon. 
This metric is also used as the objective function of the optimal control problem. 
The second metric is the total fuel consumption, where the consumption of each vehicle is evaluated as in \cref{eq:energy_integrated_out}. 
Here we report the percentage improvement in each of these two metrics relative to the baseline of an all human-driven (all following Bando-FtL model) platoon of the same size.

\subsubsection{Experiment I: Penetration Rate of AVs} In these experiments, we investigate the effect of varying the AVs penetration rate on the performance of the platoon. 
Namely, we consider a fixed platoon size of 20 vehicles and consider 5 different scenarios by introducing an increasing number of AVs from 1 ($5\%$ penetration) to 5 ($25\%$ penetration). 
The first AV is placed directly behind the leader and each additional AV is separated from its proceeding AV by 3 HV as shown in \cref{fig:2av_platoon}. 

\begin{figure}[ht]
    \centering
    \includegraphics[width=\textwidth]{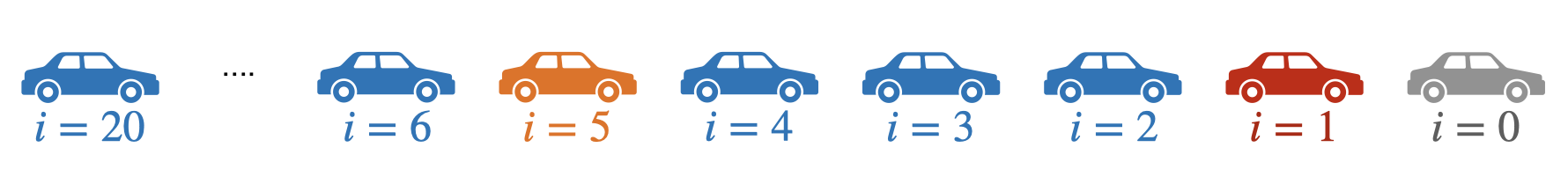}
    \caption{Illustration of a the positions of the first 2 AVs (red and orange) in a platoon of 20 vehicles. }
    \label{fig:2av_platoon}
\end{figure}

As mentioned earlier, the proposed numerical approach is only guaranteed to converge locally and it is essential to find a good initialization for the optimization algorithm. 
To that end we initialize the control function of the first AV using the solution of the problem instance in \cref{eq:ocp_instance} with a single AV following the given leader trajectory. 
We call this $\bu_{\text{init}}(t)$.  
Note that with a single AV and no HVs in the platoon the problem can be solved efficiently due to the linearity of the constraints.  
The control of the AVs in each subsequent scenario is initialized with the solution of the corresponding AV in the previous scenario and the new additional AV is initialized with $\bu_{\text{init}}(t)$.

\begin{center}
\begin{table}[tb]
\centering
\begin{tabular}{|c c c c c|}
 \hline
 Platoon & Acceleration & \% Reduction & Fuel consumption & \% Reduction \\
 \hline
 20 HV & 2381.6934 & - & 10740.3 & - \\ 
 \hline
 1 AV, 19 HV & 704.39 & 70.42\% & 8991.2 & 16.29\% \\
 \hline
 2 AV, 18 HV & 546.32 & 77.06\% & 8442.8 & 21.39\% \\
 \hline
 3 AV, 17 HV & 433.75 & 81.78\% & 8117.3 & 24.42\% \\
 \hline
 4 AV, 16 HV & 392.72 & 83.51\% & 7994.6 & 25.56\% \\ 
 \hline
 5 AV, 15 HV & 366.79 & 84.60\% & 7948.5 & 26.14\% \\ 
 \hline
\end{tabular}
\captionof{table}{Summery of the total acceleration and total fuel consumption of the vehicles in a platoon with increasing penetration rate of AVs. The percentage reduction is computed relative to the baseline of a full human-driven platoon (first row).}
\label{tab:penetration_empirical}
\end{table}
\end{center}

We summarize the performance of the 5 scenarios in terms of the total fuel consumption and total acceleration in \cref{tab:penetration_empirical}. 
In \cref{fig:penetration_1av} we show the trajectory of the platoon with a single AV. 
We notice, qualitatively from the figure, that the introduction of the AV has a significant impact on reducing the acceleration and velocity fluctuations of the entire platoon. 
In some regions the AV can almost fully smooth the fluctuation in the leader's acceleration (see the interval around $t = 200s$). 
With larger fluctuations in the leader's acceleration, the AV can reduce the magnitude of these fluctuations significantly (see the region around $t = 300s$). 
This smoothing effect propagates to the HV following the AV. 
%
%
Focusing again on the interval around $t = 300s$ (the largest peak in the lead vehicle's acceleration), the largest magnitude of the HVs acceleration is around $1.5$ times the largest magnitude of the AV in that interval. 
We compare this behavior to the baseline in \cref{fig:hv_empirical}, where the maximum HV's declaration is around $3$ times that of the lead vehicle. We make the observation that the AV's effect is not just shifting the acceleration profile but also dampening it. 
The effect of the AV is quantified in \cref{tab:penetration_empirical} showing that  introducing the AV reduces the $L^2$-norm of the platoon's acceleration by $70.42\%$ and its fuel consumption by $16.29\%$.
\begin{figure}[ht]
    \centering
    \includegraphics[width=\textwidth]{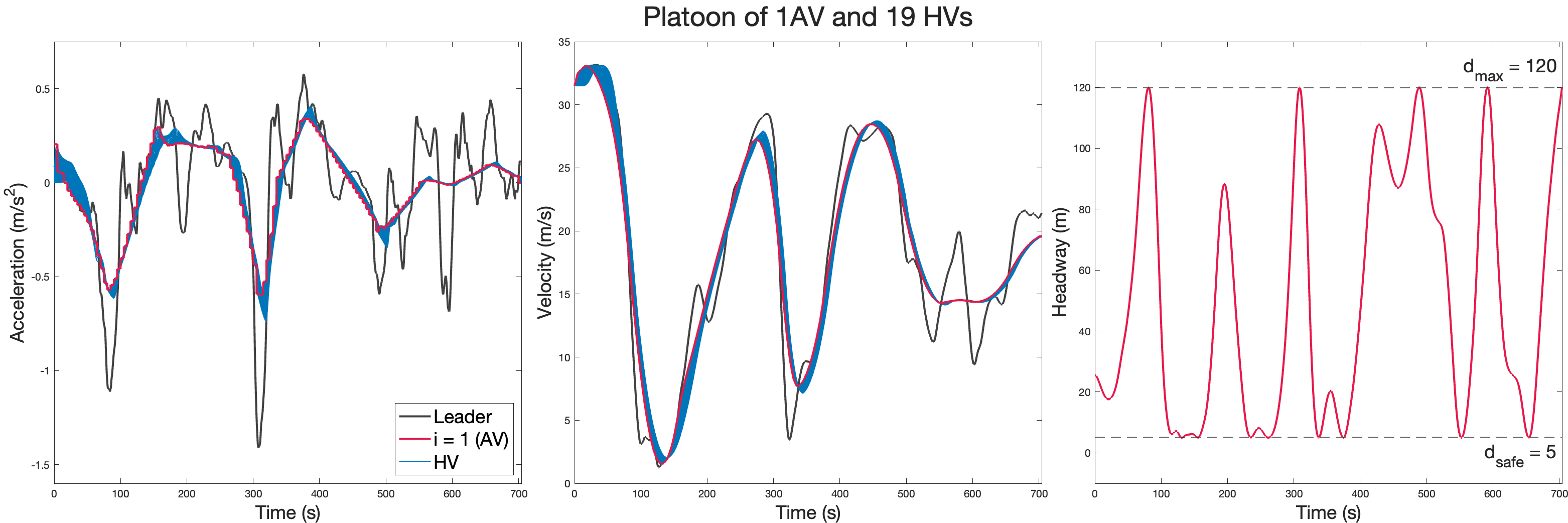}
    \caption{Acceleration (left) and velocity (middle) profiles of the 20 vehicles in the platoon with 1AV. The right tile shows the space headway of the AVs in the platoon.}
    \label{fig:penetration_1av}
\end{figure}

One thing we note from \cref{fig:penetration_1av} is that the effect of the AV diminishes for further followers. 
In \cref{fig:penetration_2av} we show the platoon trajectories with a second AV separated by 3 HV form the first one. 
The introduction of this AV further dampens the acceleration of the platoon and achieves a reduction of $77.06\%$ in the acceleration and $21.39\%$ in the fuel consumption. 
\begin{figure}[ht]
    \centering
    \includegraphics[width=\textwidth]{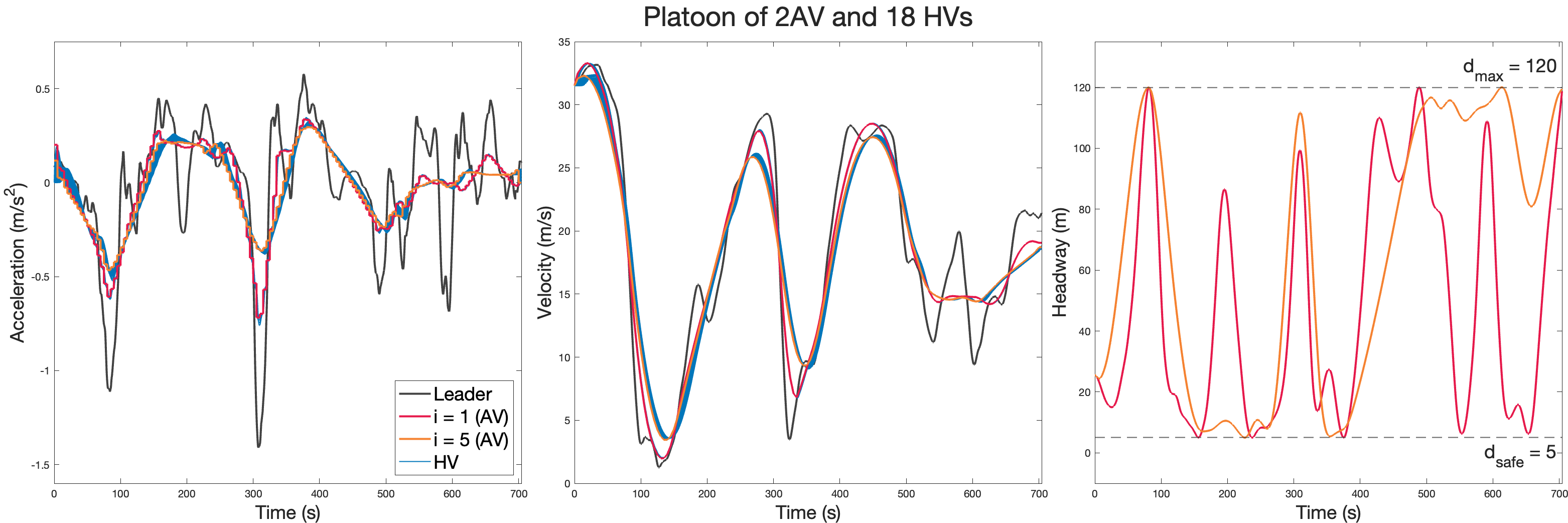}
    \caption{Acceleration (left) and velocity (middle) profiles of the 20 vehicles in the platoon with 2AV. The right tile shows the space headway of the AVs in the platoon.}
    \label{fig:penetration_2av}
\end{figure}

As we increase the penetration rate of the AVs, both the total acceleration and fuel consumption reduce, however, there is a clear diminishing return for adding more AVs. 
We note from \cref{fig:penetration_5av}, which shows the trajectories of the platoon with 5 AVs, that the last AV is still far from having a steady speed. 
Further, we notice that the maximum headway constraints are not active anywhere for this vehicle indicating that the last vehicles' behavior might not be fully optimized. 
One possible explanation for this is due to ill-conditioning of the problem. 
Notice that the further upstream the AV is located the less sensitive the objective function (and the gradients) become to changes in its control signal. 
One possible solution that we do not explore here is to use a time-weighted objective function.
\begin{figure}[ht]
    \centering
    \includegraphics[width=\textwidth]{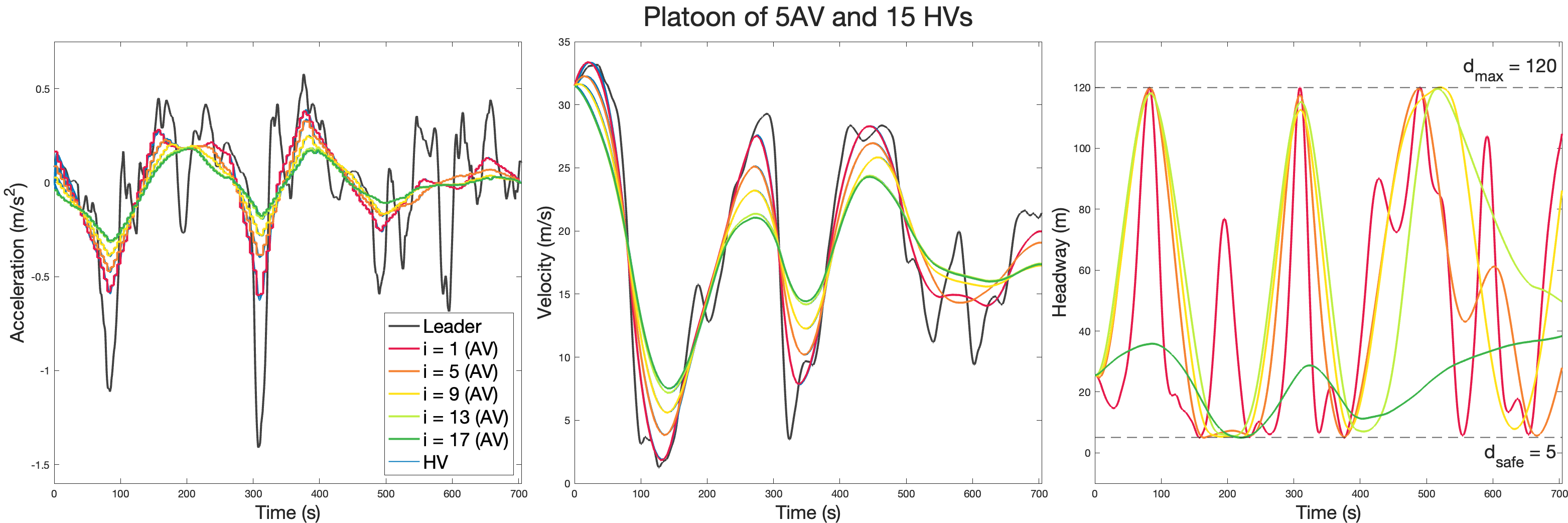}
    \caption{Acceleration (left) and velocity (middle) profiles of the 20 vehicles in the platoon with 5AV. The right tile shows the space headway of the AVs in the platoon.}
    \label{fig:penetration_5av}
\end{figure}

Lastly, we comment on the general qualitative behavior of the AVs. 
We note from \cref{fig:penetration_1av,,fig:penetration_2av,,fig:penetration_5av} that the AVs smooth the leader's trajectory by proactively and gradually slowing down and opening large gaps in anticipation of large drops in the velocity of the vehicle in front of it. 
This allows the vehicle to maintain a more steady speed profile and reduce the fluctuations in its acceleration.
In these plot we also note that all AVs respect the safety headway and maximum headway constraints. 
Further, we note that in all experiments the acceleration of the AVs is bounded by the minimum and maximum acceleration of the leader.

\subsubsection{Experiment II: Greedy Optimization}
In this experiment we consider a ``greedy'' version of the optimization problem introduced in \cref{eq:ocp_instance}. 
In this version the objective function is changed to consider the behavior of the AVs only. 
Namely, we optimize the following objective function 
\begin{align}
    J(\bx, \bv, \bu) = \sum_{i\in I_a} \int_{0}^{T} \bu_i^2(t) \dd t, 
\end{align}
While this optimization problem still depends on the car following model of choices, it puts less emphasis on optimizing its performance.
Through this approach, we obtain results that are robust to the choice of the car following model at the expense of slight performance degradation.  

\begin{center}
\begin{table}[ht]
\centering
\begin{tabular}{|c c c c|}
 \hline
 Platoon & AV Acceleration & Total acceleration & Total fuel consumption \\ [0.5ex] 
 \hline
 1 AV, 19 HV & -1.4\% & +1.4\% & +1.1\% \\
 \hline
 3 AV, 17 HV & -0.9\% & +2.4\% & +1.3\% \\
 \hline
 5 AV, 15 HV & +11.6\% & +11.6\% & +2\% \\ 
 \hline
\end{tabular}
\captionof{table}{Comparison between the platoon performance metrics under greedy vs.\ full optimization. The percentages show relative increase in the metrics value due to the greedy optimization (negative percentages indicate that the greedy optimization performs better).}
\label{tab:greedy_empirical}
\end{table}
\end{center}

In \cref{tab:greedy_empirical}, we compare the performance metrics of the platoon in which the AVs are greedily optimized to the fully optimized platoon in \cref{tab:penetration_empirical}.
The table reports the relative difference in the performance metrics due to the greedy optimization. 
\begin{align*}
    100 \times \frac{\text{greedy optimization metrics} - \text{full optimization metrics}}{\text{greedy optimization metrics}}
\end{align*}
For small penetration rates we notice very slight degradation in both the total acceleration and fuel consumption when comparing the greedy optimization to the full platoon optimization. 
For larger penetration rates, we notice a more signification degradation in the performance under the greedy optimization scheme. 
Further, we notice that even the acceleration of the AVs is higher (worse) in the greedy solution. 
This highlights a drawback of such a greedy optimization approach, when the following vehicles are not explicitly considered, cooperation between the different AVs can be less effective. 

\subsection{Conclusion} In this work, we formulated the problem of controlling the acceleration of autonomous vehicles in a mixed autonomy platoon as an optimal control problem. We analyzed the properties of the posed problem, demonstrating the well-posedness of the car-following model—a key requirement for proving the existence of a solution. This analysis can extend such results to other car-following models, provided they have such well-posedness properties.

We also develop a numerical scheme using gradient-based optimization solvers, employing the adjoint method to evaluate the derivatives and using a penalty approach to handle the state constraints. Looking forward, there are several avenues for improving the numerical methods employed in this work. First, applying higher-order optimization methods could further enhance convergence rates and solution accuracy. Second, the incorporation of Lagrange multipliers for handling the problem's constraints could provide additional theoretical insight and help us understand how these constraints affect the numerical performance of the solver. Additionally, an interesting direction for future work involves deriving the optimality conditions of the problem and applying indirect solution methods to solve the resulting system. This can potentially enhance the accuracy of the solution and overcome numerical difficulties resulting from the insensitivity of the objective function to some control variables. Overall, exploring ways to enhance numerical performance by both increasing the accuracy of the solution and speeding up the computations remains an important challenge.

Numerically, we assessed the impact of introducing automation in a predominantly human-driven platoon. Our experiments indicate significant improvements in two critical metrics: total acceleration and total fuel consumption. A natural next step is to conduct a systematic benchmark against existing longitudinal control methods to evaluate the performance of our optimized solution relative to other state-of-the-art approaches.

In summary, while this work provides a solid theoretical foundation for the optimal control of car-following models, future contributions will focus on refining and analyzing the numerical methods and improving the computational efficiency to ensure practical implementation in real-world applications.

\newpage
\bibliographystyle{abbrv}
\bibliography{bibliography} 

\end{document}